\documentclass[11pt, reqno]{amsart}

\usepackage{amsmath, amsfonts, amssymb, amsbsy, bigstrut, graphicx, enumerate,  upref, longtable, comment, comment}

\usepackage{pstricks}

\usepackage{tikz}

\usepackage[breaklinks]{hyperref}

\usepackage[numbers, sort&compress]{natbib}

\newtheorem{thm}{Theorem}[section]
\newtheorem{lmm}[thm]{Lemma}
\newtheorem{cor}[thm]{Corollary}

\theoremstyle{definition}

\newcommand{\ee}{\mathbb{E}}

\newcommand{\ma}{\mathcal{A}}

\newcommand{\mf}{\mathcal{F}}

\newcommand{\pp}{\mathbb{P}}

\newcommand{\ra}{\rightarrow}
\newcommand{\rr}{\mathbb{R}}
\newcommand{\smallavg}[1]{\langle #1 \rangle}

\newcommand{\xp}{X^\prime}

\newcommand{\zz}{\mathbb{Z}}

\newcommand{\mb}{\mathcal{B}}

\numberwithin{equation}{section}

\newcommand{\eq}[1]{\begin{align*} #1 \end{align*}}

\newcommand{\tv}{d_{\textup{TV}}}

\begin{document}
\title{Wilson loops in Ising lattice gauge theory}
\author{Sourav Chatterjee}
\address{\newline Department of Statistics \newline Stanford University\newline Sequoia Hall, 390 Serra Mall \newline Stanford, CA 94305\newline \newline \textup{\tt souravc@stanford.edu}}
\thanks{Research partially supported by NSF grant DMS-1608249}

\begin{abstract}
Wilson loop expectation in 4D $\mathbb{Z}_2$ lattice gauge theory is computed to leading order in the weak coupling regime. This is the first example of a rigorous theoretical calculation of Wilson loop expectation in the weak coupling regime of a 4D lattice gauge theory. All prior results are either inequalities or strong coupling expansions.
\end{abstract}

\maketitle

\tableofcontents

\section{Introduction}\label{isingsec}
Euclidean Yang--Mills theories are models of random connections on principal bundles that arise in quantum field theory. Lattice gauge theories are discrete approximations of Euclidean Yang--Mills theories.  A lattice gauge theory is characterized by its gauge group, and a parameter called the coupling constant. When the value of the coupling constant is small, we say that the theory is in the weak coupling regime, and when it is large, we say that the theory is in the strong coupling regime.

Wilson loop expectations are key quantities of interest in the study of lattice gauge theories, which represent discrete approximations of the integrals along curves of the random connections from Euclidean Yang--Mills theories.  Rigorous mathematical calculation of Wilson loop expectations is still mostly out of reach in dimension four, which is the dimension of greatest importance since spacetime is four-dimensional. The only cases where some rigorous approximations exist in 4D are strongly coupled theories, which are amenable to series expansions~\cite{chatterjee15, cj16, bg16, jafarov16}. In the weak coupling regime, only upper and lower bounds  are known~\cite{guth80, frohlichspencer82}. This paper gives a first-order approximation for Wilson loop expectations in weakly coupled 4D lattice gauge theory with gauge group $\zz_2$, also known as Ising lattice gauge theory. This is possibly the first instance of an explicit calculation of Wilson loop expectations in the weak coupling regime of any four-dimensional lattice gauge theory. (For a general introduction to open problems in lattice gauge theories, see~\cite{chatterjee18}.)

\subsection{Main result}
Let us begin with the definition of 4D Ising lattice gauge theory. A nice feature of the theory is that it is quite simple to describe as a model of statistical mechanics. For each $N$, let 
\[
B_N := [-N,N]^4\cap \zz^4.
\]
Let $E_N$ be the set of undirected nearest-neighbor edges of $B_N$, and let 
\[
\Sigma_N := \{-1,1\}^{E_N}
\]
be the set of all configurations of $\pm1$-valued spins assigned to edges of $B_N$. A plaquette in $\zz^4$ is a square bounded by four edges. Let $P_N$ be the set of all plaquettes whose edges are in $E_N$. If $p\in P_N$ is a plaquette with edges $e_1,e_2,e_3,e_4$, and $\sigma\in \Sigma_N$, define
\begin{equation}\label{plaquettedef}
\sigma_p := \sigma_{e_1}\sigma_{e_2}\sigma_{e_3}\sigma_{e_4}. 
\end{equation}
Ising lattice gauge theory on $B_N$ with coupling constant $g$ and free boundary condition is the probability measure on $\Sigma_N$ with probability mass function proportional to $e^{-\beta H_N(\sigma)}$ 
where $\beta := 1/g^2$ and 
\begin{equation}\label{hamdef}
H_N(\sigma) := -\sum_{p\in P_N} \sigma_p.
\end{equation}
Although $\beta$ is the square of the inverse coupling constant, we will abuse terminology and refer to $\beta$ as the {\it inverse coupling strength} in the rest of the manuscript.

Ising lattice gauge theory was in fact the first lattice gauge theory  to be defined~\cite{wegner71}, a few years before the general definition of lattice gauge theories~\cite{wilson74}. It has been investigated by physicists as a toy model for understanding phase transitions in Yang--Mills theories \cite{balianetal, fradkinshenker79, brezindrouffe82}. A few rigorous results are also known~\cite{borgs84, mm79, druhlwagner82, fm83}. In particular, a careful study with generalizations to lattice gauge theories with arbitrary finite Abelian gauge groups was conducted in~\cite{borgs84}.

Let $\gamma$ be a loop in $B_N$, with edges $e_1,\ldots, e_m$. If no edge is repeated, we will say that $\gamma$ is a self-avoiding loop. We will say that two loops are disjoint if they do not share any common edge. A finite collection of disjoint self-avoiding loops will be called a generalized loop. The length of a generalized loop $\gamma$ is defined to be the number of edges in $\gamma$. 

Given a configuration $\sigma \in \Sigma_N$ and a generalized loop $\gamma$, the Wilson loop variable $W_\gamma$ is defined as
\begin{equation}\label{wilsondef}
W_\gamma := \prod_{e\in \gamma}\sigma_{e}.
\end{equation}
Our main object of interest is the expected value of $W_\gamma$ in the lattice gauge theory defined above. This expected value will be denoted by $\smallavg{W_\gamma}_{N,\beta}$. Note that here $W_\gamma$ is a $\pm1$-valued random variable, and so its expected value lies between $-1$ and $1$. 
If $\beta$ is large enough, it turns out that the limit
\begin{equation}\label{wilsonexplim}
\smallavg{W_\gamma}_\beta := \lim_{N\to\infty} \smallavg{W_\gamma}_{N,\beta}
\end{equation}
exists for any $\gamma$ and is translation-invariant. We will prove this in Section~\ref{infproofsec}. This allows us to talk about Wilson loop expectations in Ising lattice gauge theory on the full lattice instead of finite cubes. 

Given a  generalized loop $\gamma$, an edge $e$ in $\gamma$ will be called a corner edge if there is some other edge $e'\in \gamma$ such that $e$ and $e'$ share a common plaquette. For example, a rectangular loop with length and width greater than one has exactly eight corner edges. The main result of this paper, stated below, gives a first-order approximation for Wilson loop expectation of a generalized loop when $\beta$ is large and the fraction of corner edges is small. 
\begin{thm}\label{genthm}
There exists $\beta_0>0$ such that the following holds when $\beta\ge \beta_0$. Let $\gamma$ be a nonempty generalized loop in $\zz^4$. Let $\ell$ be the number of edges in $\gamma$ and let $\ell_0$ be the number of corner edges of $\gamma$. Then 
\[
|\smallavg{W_\gamma}_\beta -e^{-2\ell e^{-12\beta}}|\le C_1\biggl(e^{-2\beta}+\sqrt{\frac{\ell_0}{\ell}}\biggr)^{C_2},
\]
where $C_1$ and $C_2$ are two positive universal constants. 
\end{thm}

For the convenience of the reader, let us interpret the above result. Suppose that $\beta$ is large and the fraction of corner edges in the loop $\gamma$ is small. Then the error bound is small. Thus, in this circumstance, if $\ell \ll e^{12\beta}$, then Theorem \ref{genthm} implies that $\smallavg{W_\gamma}_\beta\approx 1$ (which means that $W_\gamma$ is highly likely to be $1$). On the other hand, if $\ell \gg e^{12\beta}$, then we get $\smallavg{W_\gamma}_\beta\approx 0$ (which means that $W_\gamma$ is nearly equally likely to be $-1$ or $1$). Nontrivial behavior happens if and only if $\ell$ is like a constant multiple of  $e^{12\beta}$, and in that case $\smallavg{W_\gamma}_\beta\approx e^{-2\ell e^{-12\beta}}$.

Incidentally, it will be shown later (Theorem~\ref{exppropthm}) that Wilson loop expectations in 4D Ising lattice gauge theory are always nonnegative. In particular, they always lie in the interval $[0,1]$. This is not obvious, since the Wilson loop variables themselves take value in $\{-1,1\}$. The proof of Theorem \ref{exppropthm} is based on duality relations; I do not have an intuitive (or probabilistic) explanation for this positivity phenomenon. 

\subsection{Open problems}
It is mathematically quite interesting to understand higher order terms in the computation of $\smallavg{W_\gamma}_\beta$. In particular, this would shed light on what happens when $\ell\gg e^{12\beta}$ or $\ell\ll e^{12\beta}$ more precisely than what is given by Theorem~\ref{genthm}. 


A result like Theorem \ref{genthm} for non-Abelian lattice gauge theories with gauge groups such as $SU(2)$ or $SU(3)$ (and possibly with a different kind of approximation involving the area enclosed by $\gamma$) would be of great physical importance. In fact, such a result would be a significant step towards the solution of the Yang--Mills existence problem~\cite{jaffewitten, gj87, seiler82, chatterjee18}. 

The problem of extending Theorem \ref{genthm} to arbitrary finite gauge groups was posed as an open question in the original draft of this manuscript. At the time of preparing this revision, this problem has been solved by \citet{cao20} in a beautiful work that will soon be uploaded to arXiv. Cao's work covers both Abelian and non-Abelian finite groups. 

\subsection{Organization of the paper}
The rest of the paper is devoted to the proof of Theorem \ref{genthm}. The proof involves discrete exterior calculus, duality relations, and certain discrete geometrical objects called vortices. The basic framework of discrete exterior calculus is introduced in Section \ref{calcsec}. The notions of discrete surfaces and vortices are introduced in Section~\ref{strucsec}. Duality relations for Ising lattice gauge theory are derived in Section~\ref{dualsec}. Exponential decay of correlations in the weak coupling regime is proved in Section~\ref{corsec}. The nature of vortices is investigated in Section~\ref{vorsec}. Finally, the proof of Theorem~\ref{genthm} is completed in Section \ref{proofsec}.

\section{Discrete exterior calculus}\label{calcsec}
A key tool in this paper is exterior calculus for the cell complex of $\zz^n$. Although we need this tool only for $n=4$, the presentation in this section will be for general $n$. Discrete exterior calculus has been used in similar contexts in the past, for example in~\cite{frohlichspencer82} and~\cite{gross83}. An extensive survey of the methods and applications of discrete exterior calculus in duality relations for lattice models can be found in~\cite{druhlwagner82}. For the uninitiated reader, the basics are presented below. We also need a few results that may be hard to find verbatim in the literature, so those are derived here. 

\subsection{The cell complex of $\zz^n$}
Take any $n\ge 1$ and any $x\in \zz^n$. There are $n$ edges coming out of $x$ in the positive direction.  Let us denote these edges by $dx_1,\ldots, dx_n$. For each $1\le k\le n$ and $1\le i_1<i_2<\cdots<i_k\le n$, the edges $dx_{i_1},\ldots, dx_{i_k}$ define a positively oriented $k$-cell of $\zz^n$. (For example, a plaquette is a $2$-cell.)  We will denote this $k$-cell by the wedge product 
\[
dx_{i_1}\wedge dx_{i_2}\wedge \cdots \wedge dx_{i_k}.
\]
 A $0$-cell is simply a vertex. We will use $-dx_{i_1}\wedge \cdots \wedge dx_{i_k}$ to denote the negatively oriented version of $dx_{i_1}\wedge \cdots \wedge dx_{i_k}$. The collection of all oriented $k$-cells, as $k$ ranges from $0$ to $n$, is called the cell complex of $\zz^n$.

Now take any arbitrary $1\le i_1,\ldots,i_k\le n$. Following the usual conventions for wedge products, $dx_{i_1}\wedge\cdots\wedge dx_{i_k}$ is defined to be equal to zero if $i_1,\ldots, i_k$ are not all distinct, and to be equal to $(-1)^m dx_{j_1}\wedge \cdots \wedge dx_{j_k}$ otherwise, where $j_1,\ldots, j_k$ is the increasing rearrangement of $i_1,\ldots, i_k$ and $m$ is the sign of the permutation that takes $i_1,\ldots, i_k$ to $j_1,\ldots, j_k$. For example, $dx_1\wedge dx_1 = 0$ and $dx_2\wedge dx_1 = -dx_1\wedge dx_2$.

A cube $B$  in $\zz^n$ is a set of the form 
\[
([a_1,b_1]\times [a_2,b_2]\times \cdots \times [a_n, b_n])\cap \zz^n,
\] 
where the $a_i$'s and $b_i$'s are integers, such that $b_i-a_i$ is the same for each $i$. We will say that a $k$-cell $c$ is in $B$ if all the vertices of $c$ belong to~$B$. We will say that a  $k$-cell $c$ of $\zz^n$ belongs to the boundary of $B$ if every vertex of $c$ is a boundary vertex of $B$ and that $c$ is outside $B$ if at least one vertex of $c$ is outside $B$. If $c$ is in $B$ but not on the boundary of $B$, we will say that $c$ is an internal $k$-cell of $B$. 


\subsection{Differential forms on $\zz^n$}
Let $G$ be an Abelian group, with the group operation denoted by $+$, and the inverse of a group element $a$ denoted by $-a$. For simplicity, we will write $a-b$ instead of $a+(-b)$. A $G$-valued $k$-form on $\zz^n$ is a $G$-valued function on the set of positively oriented $k$-cells. A $G$-valued $k$-form $f$ will be  denoted by the formal expression
\[
f(x) = \sum_{1\le i_1<\cdots<i_k\le n} f_{i_1\cdots i_k}(x) dx_{i_1}\wedge \cdots \wedge dx_{i_k},
\]
where $f_{i_1\cdots i_k}(x)$ is the value assigned to the positively oriented $k$-cell $dx_{i_1}\wedge \cdots \wedge dx_{i_k}$. Note that a $0$-form is simply a $G$-valued function on $\zz^n$. For $k>n$ or $k<0$, there is only one $k$-form, which is denoted by $0$. If $c$ is a negatively oriented cell, $f(c)$ is defined as $-f(-c)$. 

If $B$ is a cube in $\zz^n$, then a $G$-valued $k$-form $f$ on $B$ is simply a function from the set of positively oriented $k$-cells of $B$ into $G$. 

Throughout the rest of this section, $G$ will denote an Abelian group, and $k$-forms will refer to $G$-valued $k$-forms.  Sometimes we will require $G$ to be finite. 


\subsection{Discrete exterior derivative}
For any function $h:\zz^n \to G$, any $x\in \zz^n$ and any $1\le i\le n$, define the discrete derivative of $h$ in direction $i$ at the point $x$ as
\[
\partial_i h(x) := h(x+e_i)-h(x),
\]
where $e_i$ is the vector that has $1$ in coordinate $i$ and $0$ in all other coordinates. For $0\le k\le n-1$, the exterior derivative of a $k$-form $f$ is the $(k+1)$-form
\begin{equation}
df(x)  :=  \sum_{1\le i_1<\cdots<i_k \le n} \sum_{1\le i\le n} \partial_i f_{i_1\cdots i_k}(x) dx_i \wedge dx_{i_1}\wedge \cdots\wedge dx_{i_k}.\label{dform}
\end{equation}
In other words, if $g = df$, then for any $1\le i_1<\cdots< i_{k+1}\le n$, 
\begin{equation}
g_{i_1\cdots i_{k+1}}(x) = \sum_{1\le j\le k+1} (-1)^{j-1} \partial_{i_j} f_{i_1\cdots \widehat{i_j}\cdots i_{k+1}}(x),
\label{gform}
\end{equation}
where $i_1\cdots \widehat{i_j}\cdots i_{k+1}$ is the list obtained by omitting $i_j$ from $i_1i_2\cdots i_{k+1}$. For example, if $n=4$ and 
\begin{equation}
f(x) = f_{12}(x) dx_1\wedge dx_2 + f_{13}(x) dx_1\wedge dx_3,\label{fex}
\end{equation}
then 
\begin{align*}
df(x) &= (\partial_3f_{12}(x) - \partial_2f_{13}(x)) dx_1\wedge dx_2 \wedge dx_3 \\
&\qquad+ \partial_4f_{12}(x) dx_1\wedge dx_2 \wedge dx_4  + \partial_4f_{13}(x) dx_1\wedge dx_3 \wedge dx_4.
\end{align*}
The exterior derivative of any $n$-form is zero, consistent with our convention that $0$ is the only $k$-form when $k> n$. A differential form $f$ is called closed if $df = 0$, and exact if $f = dg$ for some~$g$.

Let $B$ be a cube in $\zz^n$. If $f$ is a $k$-form on $B$, $df$ is defined the same way as before, using the formula \eqref{gform}.  This works because any $k$-cell that is contained in a $(k+1)$-cell of $B$ must be itself a $k$-cell of $B$.  Closed and exact forms on $B$ are defined as before.  

A basic result about the exterior derivative is the following lemma, which says that every exact form is closed. For completeness, the proof is provided.
\begin{lmm}\label{dlmm1}
For any $G$-valued differential form $f$, either on $\zz^n$ or on a cube $B$, $ddf = 0$.  
\end{lmm}
\begin{proof}
First suppose that $f$ is a differential form on $\zz^n$. 
If $f$ is a $k$-form for $k\ge n-1$, then $ddf =0$ automatically, so there is nothing to prove. Assume that $k< n-1$. Since the exterior derivative is a linear operator, the formula~\eqref{dform} implies that 
\eq{
ddf(x)  :=  \sum_{1\le i_1<\cdots<i_k \le n} \sum_{1\le i,j\le n}\partial_j\partial_i f_{i_1\cdots i_k}(x) dx_j\wedge dx_i \wedge dx_{i_1}\wedge \cdots\wedge dx_{i_k}.
}
Inside the sum, if $i=j$, the term is zero since $dx_i \wedge dx_i=0$. If $i\ne j$, the $(i,j)$-term cancels the $(j,i)$-term, since $dx_j \wedge dx_i = -dx_i \wedge dx_j$ and for any~$h$,
\eq{
\partial_j \partial_i h(x) &= \partial_i h(x+e_j)-\partial_i h(x)\\
&= (h(x+e_j+e_i)-h(x+e_j)) - (h(x+e_i)-h(x))\\
&= (h(x+e_j+e_i)-h(x+e_i)) - (h(x+e_j)-h(x))\\
&= \partial_i \partial_j h(x).
}
Thus, all terms cancel, showing that $ddf = 0$.

If $f$ is a differential form on a cube $B$, the same proof goes through. This is because any $k$-cell that is contained in a $(k+2)$-cell of $B$ must itself be a $k$-cell of~$B$. 
\end{proof}

\subsection{Discrete Poincar\'e lemma}
An important result about the exterior derivative is the Poincar\'e lemma. We will need a version of the classical Poincar\'e lemma in our discrete setting. Again, for completeness, the lemma is stated and proved below. Another purpose of giving a proof is that there are several fine points in the following statement that will important for us later, but are hard to find in off-the-shelf versions of this lemma.
\begin{lmm}[Poincar\'e lemma]\label{dlmm2}
Take any $1\le k\le n$. Let $G$ be an Abelian group. Let $B$ be any cube in $\zz^n$.  Then the exterior derivative $d$ is a surjective map from the set of $G$-valued  $(k-1)$-forms on $B$ onto the set of $G$-valued closed $k$-forms on $B$. Moreover, if $G$ is finite and $m$ is the number of closed $G$-valued $(k-1)$-forms on $B$, then this map is an $m$-to-$1$ correspondence. Lastly, if $1\le k\le n-1$ and $f$ is a closed $k$-form that vanishes on the boundary of $B$, then there is a $(k-1)$-form $h$ that also vanishes on the boundary of $B$ and $dh =f$. 
\end{lmm}
\begin{proof}
Let $\ma$ be the set of $G$-valued  $(k-1)$-forms on $B$, and let $\mb$ be the set of $G$-valued closed $k$-forms on $B$. 
By Lemma \ref{dlmm1}, $d$ maps $\ma$ into $\mb$. 
We will first prove that for every $f\in \mb$ there exists $g\in\ma$ such that $dg = f$. 
This claim will be proved by induction on $n$, fixing $k$. Note that the induction starts from $n=k$. We will define $g$ below, in several steps. In each step, the cases $n=k$ and $n>k$ are treated separately. When treating the case $n>k$, we will assume that the lemma has already been proved for smaller values of $n$.

For simplicity of notation, let us assume without loss of generality that $B = [a,b]^n \cap \zz^n$, where $a<b$ are two integers.   For each $a\le r\le b$, let 
\[
B_r := B \cap (\zz^{n-1}\times\{r\}).
\]
Let us now define the required $(k-1)$-form $g$ on $B$. We will define $g_{i_1\cdots i_{k-1}}(x)$ for every $x\in B$ and $1\le i_1<\cdots < i_{k-1}\le n$. In some cases, $dx_{i_1}\wedge \cdots \wedge dx_{i_{k-1}}$ may not be a $(k-1)$-cell of $B$; in those cases, we will still give the definition, but it will be irrelevant. If $k=1$, then $g_{i_1\cdots i_{k-1}}(x)$ will simply mean $g(x)$, since $g$ is a $0$-form in this case.

 Take any $1\le i_1<\cdots<i_{k-1}\le n$. If $k\ge 2$ and $i_{k-1}=n$, let $g_{i_1\cdots i_{k-1}}(x)=0$ for every $x\in B$. If $i_{k-1}< n$ or $k=1$, define $g_{i_1\cdots i_{k-1}}(x)$ as follows. 
 
 First, consider $x\in B_a$. If $n=k$, let $g_{i_1\cdots i_{k-1}}(x) = 0$. 
 If $n>k$, note that $f$ restricted to $B_a$ is a $k$-form on $B_a$ satisfying $df = 0$. By the induction hypothesis, there exists some $(k-1)$-form $g'$ on $B_a$ such that $dg' = f$ in $B_a$. Define $g_{i_1\cdots i_{k-1}}(x):=g'_{i_1\cdots i_{k-1}}(x)$, which is legitimate since $i_{k-1}< n$ or $k=1$. 
 
 Next, when $x\in B_r$ for some $a< r\le b$, and $i_{k-1}<n$ or $k=1$, define
\begin{equation}
g_{i_1\cdots i_{k-1}}(x) := g_{i_1\cdots i_{k-1}} (x-e_n) + (-1)^{k-1}f_{i_1\cdots i_{k-1} n}(x-e_n),\label{gdef}
\end{equation}
by induction on $r$. 

Let $h=dg$. Take any $1\le i_1<\cdots<i_k \le n$ and $x\in B_r$ for some $a\le r\le b$, such that $dx_{i_1}\wedge \cdots \wedge dx_{i_k}$ is a $k$-cell in $B$. We will prove by induction on $r$ that 
\begin{equation}
h_{i_1\cdots i_k}(x)=f_{i_1\cdots i_k}(x). \label{hkeq}
\end{equation}
First, suppose that $r=a$. If $n=k$, then $i_1=1, i_2=2, \ldots, i_n = n$. Thus,
\eq{
h_{i_1\cdots i_k}(x) = h_{12\cdots n}(x) = \sum_{1\le i\le n} (-1)^{i-1}\partial_i g_{12\cdots \hat{i}\cdots n} (x),
}
where, as before, the notation $12\cdots \hat{i}\cdots n$ means the list obtained by omitting $i$ from the full list $12\cdots n$. 
By the definition of $g$, $\partial_i g_{12\cdots \hat{i}\cdots n} (x) = 0$ when $i\ne n$. When $i=n$, \eqref{gdef} gives
\eq{
\partial_n g_{12\cdots (n-1)} (x) =  (-1)^{n-1}f_{12\cdots n}(x),
}
completing the proof of \eqref{hkeq} when $r=a$ and $n=k$. 

Next, suppose that $r=a$ and $n>k$. If $i_k <n$, then $dx_{i_1}\wedge \cdots \wedge dx_{i_k}$ is a $k$-cell in $B_a$. Thus, the value of $h$ on this cell is the same as that of $dg'$ on this cell. But $f = dg'$ in $B_a$. Therefore, if $r=a$ and $i_k<n$, then \eqref{hkeq} holds. 

Suppose that $r=a$ and $i_k = n$. Then
\eq{
h_{i_1\cdots i_k}(x) = \sum_{1\le j\le k}(-1)^{j-1}\partial_{i_j} g_{i_1\cdots\widehat{i_j}\cdots i_k}(x).
} 
But by the definition of $g$, $\partial_{i_j} g_{i_1\cdots\widehat{i_j}\cdots i_k}(x) = 0$ for $j<k$ (since $i_k=n$), and by \eqref{gdef},
\eq{
\partial_n g_{i_1\cdots i_{k-1}}(x)= (-1)^{k-1} f_{i_1\cdots i_{k-1} n}(x).
}
(Note that all this is valid even when $k=1$.) This completes the proof of \eqref{hkeq} when $x\in B_a$. 

Next, take $r>a$ and suppose that \eqref{hkeq} has been proved for all $x\in B_{r-1}$. Take $x\in B_r$ and any $1\le i_1<\cdots <i_k\le n$. If $i_k=n$ and $r< b$, the proof of \eqref{hkeq} is exactly the same as for $i_k =n$ and $r=a$. If $i_k=n$ and $r=b$, then $dx_{i_1}\wedge \cdots \wedge dx_{i_k}$ is not a $k$-cell in $B$, so we do not have to worry about it. This completes the proof of \eqref{hkeq} when $i_k=n$, for any $r$. 

Finally, consider the case $i_k < n$ and $r>a$. Let $y = x-e_n$, so that $y\in B_{r-1}$. Let $u=dh$ and $v = df$. Then by Lemma \ref{dlmm1}, $u= 0$, and by assumption, $v = 0$. Moreover, $dy_{i_1}\wedge \cdots \wedge dy_{i_k}\wedge dy_n$ is a $(k+1)$-cell in~$B$. Thus,
\begin{equation}
0=u_{i_1\cdots i_k n}(y) =  (-1)^k\partial_n h_{i_1\cdots i_k}(y)  + \sum_{1\le j\le k} (-1)^{j-1}\partial_{i_j} h_{i_1\cdots \widehat{i_j}\cdots i_k n}(y),\label{poin1}
\end{equation}
and 
\begin{equation}
0=v_{i_1\cdots i_k n}(y) =  (-1)^k\partial_n f_{i_1\cdots i_k}(y)  + \sum_{1\le j\le k} (-1)^{j-1}\partial_{i_j} f_{i_1\cdots \widehat{i_j}\cdots i_k n}(y).\label{poin2}
\end{equation}
Take any $1\le j\le k$. Then 
\eq{
&\partial_{i_j} h_{i_1\cdots \widehat{i_j}\cdots i_k n}(y) - \partial_{i_j} f_{i_1\cdots \widehat{i_j}\cdots i_k n}(y) \\
&= (h_{i_1\cdots \widehat{i_j}\cdots i_k n}(y+e_{i_j}) - h_{i_1\cdots \widehat{i_j}\cdots i_k n}(y)) \\
&\qquad - (f_{i_1\cdots \widehat{i_j}\cdots i_k n}(y+e_{i_j}) - f_{i_1\cdots \widehat{i_j}\cdots i_k n}(y)).
}
Since $i_k< n$, the point $z := y+e_{i_j}$ belongs to $B_{r-1}$, just like $y$. Moreover, it is easy to see that the $k$-cell $dz_{i_1}\wedge\cdots \wedge\widehat{dz_{i_j}}\wedge \cdots  \wedge dz_{i_k}\wedge dz_n$ belongs to $B$, just like $dy_{i_1}\wedge\cdots \wedge\widehat{dy_{i_j}}\wedge \cdots  \wedge dy_{i_k}\wedge dy_n$. (Here $dz_{i_1} \wedge \cdots \wedge \widehat{dz_{i_j}}\wedge \cdots \wedge dz_{i_k} \wedge dz_{i_n}$ means the wedge product $dz_{i_1} \wedge  \cdots \wedge dz_{i_k}\wedge dz_{i_n}$ with the term $dz_{i_j}$ omitted.) Therefore by the identity \eqref{hkeq} for points in $B_{r-1}$, the above expression vanishes. Thus, by \eqref{poin1} and \eqref{poin2}, we get
\eq{
\partial_n f_{i_1\cdots i_k}(y) = \partial_n h_{i_1\cdots i_k}(y)
}
But $f_{i_1\cdots i_k}(y) = h_{i_1\cdots i_k}(y)$ by \eqref{hkeq} for points in $B_{r-1}$. Since $y+e_n = x$, the above identity therefore reduces to \eqref{hkeq} for points in $B_r$. This completes the proof of the claim of that $d$ is surjective from $\ma$ onto $\mb$. 

To show that $d$ is $m$-to-$1$ from $\ma$ into $\mb$ when $G$ is finite and $m$ is defined as in the statement of the lemma, take any $f\in \mb$ and $g\in \ma$  such that $dg=f$. If $g'$ is another such $(k-1)$-form, then $d(g-g')=0$. So $g' = g + h$ for some $h$ on $B$ such that $dh =0$. Conversely, if $g' = g+h$ for some $h$ such that $dh = 0$, then $dg' = dg = f$. Thus, $dg' =f$ if and only if $g' = g + h$ for some $h$ on $B$ such that $dh=0$. 

Finally, suppose that $1\le k\le n-1$ and $f$ is a closed $k$-form on $B$ that vanishes on the boundary of $B$. We will now prove by induction on $n$ that $f=dh$ for some $(k-1)$-form $h$ that vanishes on the boundary of~$B$.

Define $g$ as before. Since $f$ vanishes on the boundary, we can choose $g=0$ on $B_a$ in the first step of the construction.  The iterative construction~\eqref{gdef} and the vanishing of $f$ on the boundary ensure that $g=0$ on every $(n-1)$-dimensional face of $B$ except $B_b$. On $B_b$, $f=dg=0$. Moreover, on the $(n-2)$-dimensional boundary of $B_b$, $g=0$ since any $(k-1)$-cell on this boundary also belongs to one of the other faces of $B$. 

If $k=1$, then $g$ is a $0$-form. In this case it follows trivially from the above observations that $g=0$ on $B_b$, allowing us to take $h=g$. This also completes the proof for $n=2$, since $k=1$ is the only possibility there. So suppose that $n>2$ and $k>1$. Then by the induction hypothesis, $g=dw$ for some $(k-2)$-form $w$ on $B_b$ which vanishes on the boundary of $B_b$. Extend $w$ to a $(k-2)$-form on $B$ by defining it to be zero outside $B_b$. Let $h = g-dw$, so that $dh=dg=f$ and $h$ vanishes on $B_b$. Now notice that $w$ vanishes on each of the other faces of $B$, because $w$ vanishes on the boundary of $B_b$ and the other faces intersect $B_b$ only at the boundary. Since $g$ vanishes on each of the other faces of $B$, this shows that $h=0$ on the boundary of $B$.
\end{proof}

\subsection{Discrete coderivative}
The exterior derivative operator $d$ has an adjoint, denoted by $\delta$ and sometimes called the `codifferential operator' or simply `coderivative'. Letting 
\[
\bar{\partial}_i h(x) := h(x)-h(x-e_i),
\]
the operator $\delta$ is defined as
\begin{equation}
\delta f(x) :=  \sum_{1\le i_1<\cdots<i_k \le n}\sum_{1\le l\le k} (-1)^l\bar{\partial}_{i_l} f_{i_1\cdots i_k}(x)dx_{i_1} \wedge \cdots \wedge \widehat{dx_{i_l}}\wedge \cdots \wedge dx_{i_k}.\label{deltadef}
\end{equation}
The operator $\delta$ takes $k$-forms to $(k-1)$-forms. If $k=0$, then $\delta f = 0$ by definition. As an example, suppose that $n=4$ and $f$ is the $2$-form defined in \eqref{fex}. Then
\begin{align*}
\delta f(x) = (\bar{\partial}_2f_{12}(x)+\bar{\partial}_3f_{13}(x))dx_1 - \bar{\partial}_1f_{12}(x)dx_2 - \bar{\partial}_1f_{13}(x)dx_3.
\end{align*}
Note that on manifolds, $\partial_i$ and $\bar{\partial}_i$ are the same. But in the discrete setting, the two operators are not the same, and we need to define $\delta$ using $\bar{\partial}_i$ as above.

\subsection{Discrete Hodge dual}
An important object for us is the dual of the lattice $\zz^n$, which we will denote by $*\zz^n$. This is just another copy of $\zz^n$, whose vertices are the centers of the $n$-cells of the primal lattice. This gives a correspondence between the $n$-cells of the primal lattice and the $0$-cells of the dual lattice. This correspondence extends to a correspondence between the cell complex of the primal lattice and the cell complex of the dual lattice, where an oriented $k$-cell $c$ of the primal lattice is paired with an oriented $(n-k)$-cell $*c$ of the dual lattice. The cell $*c$ is called the Hodge dual of the cell $c$. The duality is defined as follows.

Take any $x$ in the primal lattice. Let $y$ denote the center of the $n$-cell $dx_1\wedge\cdots \wedge dx_n$, so that $y = *(dx_1\wedge\cdots\wedge dx_n)$ according to our convention. Let $dy_1, \ldots, dy_n$ be the  edges coming out of $y$ in the {\it negative} direction. 
Take any $1\le i_1<\cdots< i_k \le n$. Let $j_1,\ldots, j_{n-k}$ be an enumeration of $\{1,\ldots, n\}\setminus \{i_1,\ldots, i_k\}$, 
and define
\[
*(dx_{i_1}\wedge \cdots \wedge dx_{i_k}) := s\,dy_{j_1}\wedge \cdots \wedge dy_{j_{n-k}},
\]
where $s$ is the sign of the permutation $(i_1,\ldots, i_k, j_1,\ldots, j_{n-k})$. 
It is easy to see that the right side does not depend on our choice of $j_1,\ldots, j_{n-k}$. By the same principle, define
\[
*(dy_{j_1}\wedge \cdots \wedge dy_{j_{n-k}}) := (-1)^{k(n-k)}s\,dx_{i_1}\wedge \cdots \wedge dx_{i_k},
\]
where $(-1)^{k(n-k)}$ is present because the sign of $(j_1,\ldots, j_{n-k}, i_1,\ldots, i_k)$ is equal to $(-1)^{k(n-k)}s$. For example, if $n=4$,  $*(dx_1 \wedge dx_2) = dy_3\wedge dy_4$ and $*(dx_1\wedge dx_3) = -dy_2\wedge dy_4$. 
With $x$, $y$, $s$ and $j_1,\ldots, j_{n-k}$ as above, the Hodge dual of a $G$-valued $k$-form $f$ is defined as the following $(n-k)$-form on the dual lattice:
\[
*f(y) := \sum_{1\le i_1<\cdots< i_k\le n} f_{i_1\cdots i_k} (x)  s\, dy_{j_1}\wedge \cdots \wedge dy_{j_{n-k}}.
\]
(Note that $s$ and $j_1,\ldots, j_{n-k}$ depend on $i_1,\ldots, i_k$ in the above sum, but this dependence has been suppressed for notational clarity.) 
For example, if $n=4$ and $f$ is given by \eqref{fex}, then
\[
*f(y) = f_{12}(x) dy_3\wedge dy_4 - f_{13}(x) dy_2 \wedge dy_4.
\]
It is easy to check that for any $f$, $*f(*c)=f(c)$ and ${*{*f}} = (-1)^{k(n-k)} f$. 

The exterior derivative operator on the cell complex of the dual lattice is defined just the same way as on the primal lattice, but with $\partial_i$  replaced by $\bar{\partial}_i$, since the edges in the dual lattice are oriented in the opposite direction. The exterior derivative on the dual lattice is also denoted by $d$. With this definition, the following lemma connects the Hodge star operator with the adjoint of the exterior derivative on the dual lattice. This is a discrete version of a well-known lemma about differential forms on manifolds.
\begin{lmm}\label{dlmm3}
For any $G$-valued $k$-form $f$ on $\zz^n$, 
\eq{
\delta f(x) = (-1)^{n(k+1)+1} {*d{*f}}(y),
}
where $y$ is center of the $n$-cell $dx_1\wedge\cdots \wedge dx_n$.
\end{lmm}
\begin{proof}
Note that if $f$ is a $G$-valued $k$-form on the primal lattice, and $x$, $y$, $s$ and $j_1,\ldots, j_{n-k}$ are as in the paragraphs preceding the statement of the lemma, then
\eq{
d{*f}(y) = \sum_{1\le i_1<\cdots< i_k\le n} \sum_{1\le i\le n} \bar{\partial}_if_{i_1\cdots i_k} (x) s\,dy_i\wedge dy_{j_1}\wedge \cdots \wedge dy_{j_{n-k}}.
}
The summand on the right is zero unless $i$ is one of $i_1,\ldots, i_k$.  If $i=i_l$ for some $l$, then 
\eq{
&{*(dy_i\wedge dy_{j_1}\wedge \cdots \wedge dy_{j_{n-k}})} \\
&= (-1)^{k(n-k)+n-k+l-1}s\, dx_{i_1}\wedge\cdots \wedge\widehat{dx_{i_l}} \wedge \cdots \wedge dx_{i_k},
}
since  the sign of $(i_l, j_1,\ldots, j_{n-k}, i_1,\ldots,\widehat{i_l}, \ldots, i_k)$ equals 
\[
(-1)^{k(n-k) + n-k+l-1}s.
\]
Thus, putting $r := k(n-k)+n-k-1$, we get
\eq{
&{*d{*f}}(y) \\
&= (-1)^{r}\sum_{1\le i_1<\cdots<i_k \le n}\sum_{1\le l\le k} (-1)^l\bar{\partial}_{i_l} f_{i_1\cdots i_k}(x)dx_{i_1} \wedge \cdots \wedge \widehat{dx_{i_l}}\wedge \cdots \wedge dx_{i_k}\\
&= (-1)^{r}\delta f(x).
}
Since $k(k+1)$ is even, $(-1)^r = (-1)^{n(k+1)+1}$. 
\end{proof}

\subsection{Hodge dual on a cube}
For any cube $B$ in $\zz^n$, define the dual cube $*B$ to be the union of all $n$-cells in the dual lattice that are duals of the vertices of $B$. It is not difficult to see that the dual of the cube 
\[
B = ([a_1,b_1]\times\cdots \times[a_n, b_n])\cap \zz^n
\]
is the cube 
\[
{*B} = ([a_1-1/2, b_1+1/2]\times \cdots \times[a_n-1/2, b_n+1/2])\cap *\zz^n,
\]
and 
\[
{*{*B}} = ([a_1-1, b_1+1]\times \cdots \times[a_n-1, b_n+1])\cap \zz^n.
\]
Recall that a $k$-cell $c$ is in $B$ if all the vertices of $c$ are vertices of $B$. We will say that a  $k$-cell $c$ belongs to the boundary of $B$ if every vertex of $c$ is a boundary vertex of $B$. We will say that $c$ is outside $B$ if at least one vertex of $c$ is outside $B$. If $c$ is in $B$ but not on the boundary of $B$, we will say that $c$ is an internal $k$-cell of $B$. 
\begin{lmm}\label{dlmm4}
Let $B$ be any cube in $\zz^n$. Then a $k$-cell $c$ is outside $B$ if and only if $*c$ is either outside $*B$ or on the boundary of $*B$. Moreover, if $c$ is a $k$-cell outside $B$ that contains a $(k-1)$-cell of $B$, then $*c$ belongs to the boundary of $*B$.
\end{lmm}
\begin{proof}
Without loss of generality, suppose that $B = [a,b]^n \cap \zz^n$ for some $a<b$. First, suppose that $c$ is in $B$. Then  all the vertices of $*c$ belong to $*B$, and therefore $*c$ is in $*B$. Moreover, one of the vertices of $*c$ is the center of an $n$-cell of $B$. The center of an $n$-cell of $B$ cannot be a boundary vertex of $*B$. Thus, $*c$ cannot belong to the boundary of $*B$. This proves the `if' part of the first claim. 

To prove the `only if' part, take a $k$-cell $c$ outside $B$.  Suppose that 
 $c = dx_{i_1}\wedge \cdots \wedge dx_{i_k}$. 
Let $y$ be the center of the $n$-cell containing $c$. Since $c$ is outside $B$, at least one of the integers $x_1,\ldots, x_n, x_{i_1}+1,\ldots, x_{i_k}+1$ must be outside the interval $[a,b]$. 

Suppose that $x_{i_1}+1\not\in [a,b]$. Then at least one of the two numbers $y_{i_1}-1$ and $y_{i_1}+1$ is not in $[a-1/2,b+1/2]$. This allows us to produce, for each vertex of $*c$, a neighbor that does not belong to $*B$. Therefore $*c$ is either on the boundary of $*B$ or outside $*B$. 

Similarly, if  $x_1\not\in [a,b]$, then either $x_1 > b$ or $x_1 < a$. In the first case, $y_1\not\in [a-1/2, b+1/2]$, and in the second case, $y_1-1\not\in [a-1/2, b+1/2]$. Again, this allows us to produce for each vertex of $*c$ a neighbor that is not in $*B$. This completes the proof of the `only if' part of the claim.

To prove the last part, let $c = dx_{i_1}\wedge \cdots \wedge dx_{i_k}$ be a $k$-cell outside $B$. If $c$ contains a $(k-1)$-cell of $B$,  then there exists $1\le l\le k$ such that $x_{i_l}+1\in [a,b]$, $x_{i_j}\in [a,b-1]$ for every $j\ne l$, and $x_i\in [a,b]$ for every $i\not\in \{i_1,\ldots,i_k\}$. Let $y_i = x_i+1/2$ for each $i$. Then  $y_i\in [a-1/2, b+1/2]$ for all $i$, and  $y_i\in [a+1/2, b+1/2]$ for all $i\not\in \{i_1,\ldots,i_k\}$. This implies that $*c$ is in $*B$. By the first part, $*c$ must therefore be on the boundary of $*B$.
\end{proof}
Suppose that $f$ is a $G$-valued $k$-form on $B$. The Hodge dual $*f$ of $f$ is an $(n-k)$-form on $*B$ defined as follows. If $c$ is a $k$-cell of $B$, let $*f(*c) := f(c)$ as usual. By Lemma \ref{dlmm4}, this defines $*f$ on all the  internal $(n-k)$-cells of $*B$. On the boundary cells, define $*f$ to be zero. Another way to say this is the following. Extend $f$ to a $k$-form $f'$ on $\zz^n$ by defining it to be zero outside $B$. Then define $*f$ to be the restriction of $*f'$ to $*B$. By Lemma~\ref{dlmm4}, the two definitions are equivalent.

The  codifferential $\delta f$ is defined using a similar principle: First extend $f$ to a $k$-form $f'$ on all of $\zz^n$ by defining it to be zero outside $B$, then define $\delta f'$ using \eqref{deltadef}, and finally let $\delta f$ be the restriction of $\delta f'$ to the $(k-1)$-cells of $B$. 
\begin{lmm}\label{dlmm5}
Let $f$ be  a $G$-valued $k$-form on a cube $B$ in $\zz^n$. If $f'$ is a $k$-form on $\zz^n$ such that $f'=f$ on $B$ and $f'(c)=0$ for any $c$ such that $*c$ is a boundary cell of $*B$, then $\delta f = \delta f'$ in $B$ and $*f = *f'$ in $*B$.
\end{lmm}
\begin{proof}
It follows directly from the way $*f$ was defined above, the given conditions on $f'$, and Lemma~\ref{dlmm4}, that $*f = *f'$ in $*B$. To prove the other assertion, extend $f$ to a $k$-form $f''$ on $\zz^n$ by defining it to be zero outside $B$.  Let $c$ be a $(k-1)$-cell of $B$. By the formula \eqref{deltadef}, $\delta f''(c)$ is a linear combination of $f''(c')$ as $c'$ ranges over all $k$-cells $c'$ containing $c$. Take any such $c'$. If $c'$ is in $B$, then $f'(c') = f''(c')$. If $c'$ is not in $B$, then by Lemma~\ref{dlmm4}, $*c'$ is a boundary cell of $*B$, which implies that $f'(c')=0 = f''(c')$. This shows that $\delta f '= \delta f'' = \delta f$ in $B$.
\end{proof}

\subsection{Primal-dual correspondence}
Lemmas \ref{dlmm1}--\ref{dlmm5} combine to yield the following important result, which will be crucial for our calculations. 
\begin{lmm}\label{dproplmm}
Take any $1\le k\le n-1$. Let $G$ be an Abelian group. Let $B$ be any cube in $\zz^n$.  Then the exterior derivative $d$ is a surjective map from the set of $G$-valued  $(n-k-1)$-forms $g$ on $*B$ that satisfy $dg = 0$ on the boundary of $*B$, to the set of duals of $k$-forms $f$ on $B$ that satisfy $\delta f = 0$. Moreover, if $G$ is finite and $m$ is the number of closed $G$-valued $(n-k-1)$-forms on $*B$, then this map is an $m$-to-$1$ correspondence.
\end{lmm}
\begin{proof}
Let $\ma$ be the set of $G$-valued  $(n-k-1)$-forms $g$ on $*B$ that satisfy $dg = 0$ on the boundary of $*B$, and let $\mb$ be the set of $k$-forms $f$ on $B$ that satisfy $\delta f = 0$.  

Take any $f\in \mb$. Extend $f$ to a $k$-form $f'$ on $\zz^n$ by setting $f'=f$ in $B$ and $f'=0$ outside $B$. Then $\delta f' = 0$ in $B$ since $f\in \mb$, and $\delta f' = 0$ outside $B$ because any $(k-1)$-cell outside $B$ can belong to only $k$-cells outside $B$, where $f'$ is zero. Thus by Lemma \ref{dlmm3}, we see that $d{*f'} = 0$ everywhere. Therefore by Lemma~\ref{dlmm2}, there exists an $(n-k-1)$-form $g$ on $*B$ such that $dg = *f'$ in $*B$, and moreover, there are exactly $m$ such $g$ if $G$ is finite. But $*f' = *f$ in $*B$, so $dg = *f$ in $*B$. Since $f' = 0$ outside $B$, and the dual of any cell on the boundary of $*B$ must be outside $B$ by Lemma \ref{dlmm4}, $dg$ must be zero on the boundary of $*B$. Thus, $g\in \ma$. This shows that the dual of every $f\in \mb$ is the image of some element of $\ma$ under the map $d$. Moreover, if $G$ is finite, then exactly $m$ elements of $\ma$ map to the dual of a given element of $\mb$ under the map $d$. 

Next, take any $g\in \ma$. Extend $g$ to a $(n-k-1)$-form $g'$ on $*\zz^n$ by setting $g' = g$ in $*B$ and $g'=0$ outside $*B$. Let $f' := (-1)^{k(n-k)}{*dg'}$. Then $dg'=*f'$ is the dual of the $k$-form $f'$.  By Lemma~\ref{dlmm1} and Lemma~\ref{dlmm3}, this shows that $\delta f' = 0$. Moreover, since $dg' =0$ on the boundary of $*B$, $f'(c)=0$ for any $c$ such that $*c$ belongs to the boundary of $*B$. Let $f$ be the restriction of $f'$ to $B$. Then by Lemma \ref{dlmm5}, $\delta f = \delta f' = 0$ in $B$ and $*f = *f' = dg' = dg$ on $*B$. This shows that $d$ maps $\ma$ into the set of duals of elements of $\mb$, completing the proof.
\end{proof}

\subsection{Poincar\'e lemma for the coderivative}
There is a dual version of the Poincar\'e lemma that will be needed for certain purposes. The proof is essentially a combination of the Poincar\'e lemma and Lemma \ref{dlmm3}. 
\begin{lmm}[Poincar\'e lemma for the coderivative]\label{poinlmm2}
Take any $1\le k\le n-1$. Let $f$ be a $G$-valued $k$-form on $\zz^n$ which is zero outside a finite region. Suppose that $\delta f = 0$. Then there is a $(k+1)$-form $h$ such that $f = \delta h$. Moreover, if $f$ is zero outside a cube $B$, then there is a choice of $h$ that is zero outside $B$.
\end{lmm}
\begin{proof}
Since $\delta f = 0$, Lemma \ref{dlmm3} implies that $d{*f} = 0$. Since $f =0$ outside a finite region, there is a large enough cube $B$ such that $f=0$ outside $B$. Let $f'$ be the restriction of $f$ to $B$. By definition, $*f'$ vanishes on the boundary of $*B$. By Lemma \ref{dlmm4} and the fact that $f$ vanishes outside $B$, $*f = 0$ on the boundary of $*B$. Therefore,  $d{*f'}=d{*f} = 0$ everywhere in $*B$. By Lemma~\ref{dlmm2}, this implies that there is a $(n-k-1)$-form $g'$ on $*B$ such that $dg' = *f'$ and $g'$ vanishes on the boundary of $*B$. Extend $g'$ to a $(n-k-1)$-form $g$ on $*\zz^n$ by defining it to be zero outside $*B$. Since $g$ vanishes on the boundary of $*B$ and outside $*B$, $dg = 0 =*f$ outside $*B$. Combining, we see that $dg=*f$ everywhere. Let 
\[
h := (-1)^{-(k-1)(n-k+1)-k(n-k)-nk-1}{*g},
\]
so that
\[
*h = (-1)^{-k(n-k)-nk-1} g.
\]
Thus, by Lemma \ref{dlmm4}, 
\eq{
\delta h &= (-1)^{nk+1}{*d{*h}}\\
&= (-1)^{-k(n-k)} {*dg} = f.
}
Lastly, by Lemma \ref{dlmm4} and the definition of $h$, it follows that $h$ vanishes outside the cube $B$. 
\end{proof}

\section{Discrete surfaces and vortices}\label{strucsec}
In this section, we will specialize the results of Section \ref{calcsec} to the lattice $\zz^4$ and the group $\zz_2$, and make some important applications of the results. In this section, $\zz_2$ will be treated as the quotient group $\zz/2\zz$, with the operation of addition modulo $2$.

\subsection{Surfaces} 
Note that since $x=-x$ in the group $\zz_2$, there is no reason to distinguish between positively and negatively oriented cells in the cell complex of $\zz^4$ when we are dealing with $\zz_2$-valued differential forms. Accordingly, edges may be identified with $1$-cells and plaquettes with $2$-cells. A unique feature of $\zz^4$ is that duals of plaquettes are plaquettes in the dual lattice. 

Since $\zz_2$ contains only two elements, there is a natural one-to-one correspondence between sets of edges and $\zz_2$-valued $1$-forms on $\zz^4$.  Namely, given a $1$-form, we associate to it the set of edges where the $1$-form takes value $1$. Similarly, there is a one-to-one correspondence between $\zz_2$-valued $2$-forms and sets of plaquettes, and a one-to-one correspondence between $0$-forms and sets of vertices.

We will refer to a set of plaquettes as a `surface'. If $P$ is surface, we will refer to the set of dual plaquettes, $*P$, as the `dual surface'.

With the above correspondence, it is easy to see that if $\gamma$ is a finite set of edges, and $f$ is the $1$-form that corresponds to $\gamma$, then the $0$-form $\delta f$ corresponds to the set of vertices that are contained in an odd number of edges of $\gamma$. In other words, $\gamma$ is a generalized loop if and only if $\delta f = 0$.

We define the boundary of a surface $P$ as the set of edges that are contained in an odd number of plaquettes in $P$. In other words, if $f$ is the $2$-form corresponding to $P$, then the boundary of $P$ is simply the set of edges corresponding to the $1$-form $\delta f$. We will say that a surface is closed if its boundary is empty. An edge of $P$ that is not a boundary edge will be called an internal edge. 

The following result is a simple consequence of the Poincar\'e lemma for the coderivative.
\begin{lmm}\label{looplmm}
If $\gamma$ is a generalized loop in $\zz^4$, then $\gamma$ is the boundary of some surface $P$. Moreover, if $\gamma$ is contained in a cube $B$, then there is a choice of $P$ that is also contained in $B$.
\end{lmm}
\begin{proof}
Let $f$ be the $1$-form corresponding to $\gamma$. Then $\delta f=0$, as observed above. Therefore by Lemma \ref{poinlmm2}, there is a $2$-form $h$ such that $\delta h = f$. Moreover, if $f$ is zero outside $B$, then $h$ can be chosen such that $h=0$ outside $B$. This completes the proof of the lemma.
\end{proof}
A plaquette $p$ will be called an internal plaquette of a surface $P$ if none of the edges of $p$ is a boundary edge of $P$. Otherwise, $p$ will be called a boundary plaquette of $P$. The following lemma is a crucial component of our argument for establishing the perimeter law in the weak coupling limit of Ising lattice gauge theory. 
\begin{lmm}\label{evenlmm}
Let $P$ and $Q$ be two surfaces, and suppose that the dual surface $*P$ is closed. If there is a cube $B$ containing $P$ such that $*{*B}\cap Q$ consists of only internal plaquettes of $Q$, then $|P\cap Q|$ is even. 
\end{lmm}
\begin{proof}
Let $f$ be the $2$-form corresponding to the surface $P$. Since $*P$ is closed, $\delta{*f}=0$. Moreover, $*f$ is zero outside $*B$ by Lemma \ref{dlmm4}. By Lemma~\ref{poinlmm2}, this implies that $*f = \delta g$ for some $3$-form $g$ that is zero outside $*B$. Therefore by Lemma \ref{dlmm3}, 
\[
f = *{*f} = d{*g}.  
\]
Let
\[
S := \sum_{p\in  Q} f(p),
\]
where the sum is in $\zz_2$. This is a finite sum since only finitely many terms are nonzero. The proof of the lemma will be complete if we can show that $S=0$. To that end, note that
\eq{
S &= \sum_{p\in Q} d{*g}(p) = \sum_{p\in Q} \sum_{e\in p} *g(e) = \sum_{e\in \gamma} *g(e),
}
where $\gamma$  is the boundary of $Q$. 

Let us say that a $3$-form on the dual lattice is `elementary' if its value is $1$ on a single $3$-cell, and $0$ elsewhere. Clearly, any $3$-form with finitely many $1$'s is a sum of elementary $3$-forms. In particular, the $3$-form $g$ is a finite sum of elementary $3$-forms that are zero outside $*B$. Since the Hodge star operator is additive, this shows that to prove that $S=0$, it suffices to prove    that 
\begin{equation}\label{g0e}
\sum_{e\in \gamma} *g_0(e)=0
\end{equation}
for any elementary $3$-form $g_0$ that is zero outside $*B$. Take any such $g_0$. Then $*g_0$ is an elementary $1$-form of the primal lattice. Let $c$ be the $3$-cell where $g_0$ is $1$ and let $e = *c$ be the edge where $*g_0$ is $1$. Let $P_0$ be the set of all plaquettes containing $e$. Then $P_0$ is the set of duals of the boundary plaquettes of $c$. In particular, the elements of $P_0$ are plaquettes of $*{*B}$. Let $Q_0$ be the set of plaquettes of $Q$ that contain $e$. Then 
\[
Q_0=P_0\cap Q\subseteq *{*B}\cap Q.
\]
Thus, the elements of $Q_0$ are all internal plaquettes of $Q$, and so $e$ cannot be a boundary edge of $Q$. Consequently, $|Q_0|$ must be even. Thus, $e\not\in \gamma$ and hence \eqref{g0e} holds. As noted before, this implies that $S=0$, completing the proof.  
\end{proof}

\subsection{Vortices}
Let $\Sigma$ be the set of all spin configurations on the edges of $\zz^4$. Given a configuration $\sigma\in \Sigma$, we will say that plaquette $p$ is a `negative plaquette' for this configuration if $\sigma_p=-1$. In the following, two plaquettes will be called adjacent if they share a common edge. A surface will be called connected if it is connected with respect to this notion of adjacency.


Given a spin configuration $\sigma$, a surface $P$ will be called a {\it vortex} if the dual surface $*P$  is closed and connected, and every member of $P$ is a negative plaquette for the configuration $\sigma$. The term `vortex' comes from  similar usage in the physics literature. Vortices have played an important role in the physics literature on quark confinement in lattice gauge theories, starting from~\cite{no73, tassie73, nambu74, parisi75, ks74, thooft74b}. 
\begin{lmm}\label{vorlmm}
For any configuration $\sigma$, the set of negative plaquettes of $\sigma$ is a disjoint union of vortices.
\end{lmm}
\begin{proof}
Let $P$ be the set of negative plaquettes. Let $f$ be the $2$-form corresponding to $P$ and $g$ be the $1$-form corresponding to $\sigma$. Note that $f=dg$. Thus, by Lemma \ref{dlmm1}, $df=0$ and hence by Lemma \ref{dlmm3}, 
\[
\delta{*f} = {*d}f = 0.
\]
The $2$-form $*f$ corresponds to the dual surface $*P$. The above identity implies that every edge of the dual lattice belongs to an even number of elements of $*P$. In other words, $*P$ is a closed surface. 

Finally, note that since there is no edge that is shared by two plaquettes in two distinct connected components of $*P$, each connected component of $*P$ must also be a closed surface. This completes the proof.
\end{proof}

\section{Duality relations}\label{dualsec}
Recall the four-dimensional Ising lattice gauge theory on $B_N$ defined in Section \ref{isingsec}. We will refer to this model as having `free boundary condition', since no condition is imposed on the spins on boundary edges.  Let $Z_N(\beta)$ be the partition function of Ising lattice gauge theory with free boundary condition on $B_N$ at inverse coupling strength $\beta$. 

There is another kind of boundary condition that will be important for us. This is the condition that $\sigma_p=1$ for all boundary plaquettes $p$ of $B_N$. We will refer to this as the zero boundary condition (since the element $1$ is the zero of the group $\zz_2$). 

\subsection{Duality for the partition function}
We will now express $Z_N(\beta)$ in terms of an Ising lattice gauge theory on the dual cube $*B_N$. Such dualities are well known to physicists. One of the earliest occurrences in the context of lattice gauge theories was in~\cite{wegner71}, where duals of Ising models were constructed using lattice gauge theories in dimensions three and higher. Duality for $U(1)$ lattice gauge  theory with Villain action has been used in several important papers; for example, in~\cite{frohlichspencer82} duality was used for the  proof of the deconfinement transition in four-dimensional $U(1)$ lattice gauge theory. For an extensive discussion of such duality relations, see~\cite{druhlwagner82}.
\begin{thm}\label{zpropthm}
Let $P_N$ be the set of plaquettes of $B_N$ and $E_N$ be the set of edges of $B_N$. Let $a_N$ be the number of closed $\zz_2$-valued $1$-forms on $*B_N$ and $b_N$ be the number of boundary plaquettes of $*B_N$.  Let
\begin{equation}
\lambda := -\frac{1}{2}\log\tanh \beta,\quad \alpha := (\cosh\beta \sinh\beta)^{1/2}, \label{lambdadef}
\end{equation}
and let $Z_N^*(\lambda)$ be the partition function of Ising lattice gauge  theory with zero boundary condition on  $*B_N$ at inverse coupling strength $\lambda$. Then
\eq{
Z_N(\beta) = \frac{2^{|E_N|}\alpha^{|P_N|}e^{-\lambda b_N}}{a_N}  Z_N^*(\lambda).
}
\end{thm}
\begin{proof}
Let $\Sigma_N := \{-1,1\}^{E_N}$. For $\sigma\in \Sigma_N$ and $p\in P_N$, let $\sigma_p$ be defined as in Section \ref{isingsec}. Then
\begin{equation}
Z_N(\beta) = \sum_{\sigma\in \Sigma_N} \prod_{p\in P_N} e^{\beta \sigma_p}\,.\label{znbeta}
\end{equation}
Since $\sigma_p$ is either $1$ or $-1$, 
\[
e^{\beta \sigma_p} = \cosh \beta + \sigma_p \sinh\beta = \alpha(e^\lambda + \sigma_p e^{-\lambda}).
\]
Substituting this in \eqref{znbeta}, we get
\begin{equation}
Z_N(\beta) = \alpha^{|P_N|}\sum_{\sigma\in \Sigma_N} \prod_{p\in P_N}(e^\lambda + \sigma_p e^{-\lambda}).\label{zexpr}
\end{equation}
Expanding the product on the right gives
\eq{
\prod_{p\in P_N}(e^\lambda + \sigma_p e^{-\lambda}) &= \sum_{\kappa\in \{0,1\}^{P_N}} \prod_{p\in P_N}\sigma_p^{\kappa_p} e^{\lambda(1-2\kappa_p)}.
}
Making the change of variable $\tau_p=1-2\kappa_p$, this gives
\eq{
\prod_{p\in P_N}(e^\lambda + \sigma_p e^{-\lambda}) &= \sum_{\tau \in \Gamma_N} \prod_{p\in P_N}\sigma_p^{(1-\tau_p)/2} e^{\lambda \tau_p},
}
where
\eq{
\Gamma_N := \{-1,1\}^{P_N}.
}
Substituting this back in \eqref{zexpr}, we get
\begin{align}
Z_N(\beta) &= \alpha^{|P_N|}\sum_{\sigma\in \Sigma_N} \sum_{\tau\in \Gamma_N}\prod_{p\in P_N} \sigma_p^{(1-\tau_p)/2}e^{\lambda \tau_p}\nonumber\\
&= \alpha^{|P_N|}\sum_{\tau\in \Gamma_N} e^{\lambda\sum_{p\in P_N} \tau_p}\Big(\sum_{\sigma\in \Sigma_N} \prod_{e\in E_N} \sigma_e^{f(\tau,e)}\Big),\label{znbeta2}
\end{align}
where
\[
f(\tau, e):= \sum_{p\in P_N, p\ni e} \frac{1-\tau_p}{2}.
\]
In other words, $f(\tau,e)$ counts the number of plaquettes $p\in P_N$ containing the edge $e$ for which $\tau_p=-1$. For convenience, let us call this set of plaquettes $P_N(e)$. The number $f(\tau, e)$ is even if and only if 
\begin{equation}
\prod_{p\in P_N(e)} \tau_p = 1.\label{prodcond}
\end{equation}
Consequently,
\begin{align}
\sum_{\sigma\in \Sigma_N} \prod_{e\in E_N} \sigma_e^{f(\tau,e)}  &= \prod_{e\in E_N} (1+ (-1)^{f(\tau,e)}) \nonumber\\
&= 
\begin{cases}
2^{|E_N|} &\text{ if  $\prod_{p\in P_N(e)} \tau_p = 1$ for every $e\in E_N$,}\\
0 &\text{ otherwise.}
\end{cases}
\label{sigmaform}
\end{align}
Let $\Gamma_N'$ be the set of all $\tau\in \Gamma_N$ for which \eqref{prodcond} holds for every $e$. By \eqref{znbeta2} and \eqref{sigmaform}, we get
\begin{equation}
Z_N(\beta) = 2^{|E_N|}\alpha^{|P_N|}\sum_{\tau\in \Gamma_N'} e^{\lambda \sum_{p\in P_N} \tau_p}.\label{znbeta3}
\end{equation}
Now note that any $\tau \in \Gamma_N$ naturally defines a $\zz_2$-valued $2$-form $f$ on  $B_N$, through the correspondence that sends spin $1$ to the element $0$ of $\zz_2$ and spin $-1$ to the element $1$ of $\zz_2$. In terms of $f$, the condition \eqref{prodcond} precisely means that the $1$-form $\delta f$ vanishes on the edge~$e$. Thus, $\tau\in \Gamma_N'$ if and only if the corresponding $f$ satisfies $\delta f = 0$. On the other hand, any $\zz_2$-valued $1$-form $g$ on $*B_N$ corresponds to  a spin configuration $\sigma$ on the set of edges of $*B_N$, and $dg=0$ on the boundary of $*B_N$ if and only if $\sigma$ satisfies the zero boundary condition. Take any such $g$ and suppose that $dg = *f$. Let $P_N^*$ denotes the set of plaquettes of $*B_N$, and let  $*P_N$ be the set of dual plaquettes of $P_N$, as usual. Note that $*P_N\subseteq P_N^*$, and by Lemma \ref{dlmm4}, $P_N^*\setminus {*P_N}$ is the set of boundary plaquettes of $*B_N$. Notice that for any $p\in *P_N$, 
\[
\frac{1-\sigma_{p}}{2} = dg(p) = *f(p)= f(*p) =\frac{1-\tau_{*p}}{2}.
\]
Thus, $\sigma_p = \tau_{*p}$ for all $p\in *P_N$. On the other hand, if $p\in P_N^*\setminus {*P_N}$, then $p$ is a boundary plaquette of $*B_N$ and hence 
\[
\frac{1-\sigma_p}{2}=dg(p)=0,
\]
implying that $\sigma_p=1$. Thus,
\[
\sum_{p\in P_N^*} \sigma_p = b_N + \sum_{p\in P_N} \tau_p.
\]
Let $*\Sigma_N$ be the set of spin configurations on the edges of $*B_N$ that satisfy the zero boundary condition. By Lemma~\ref{dproplmm}, there are exactly $a_N$ configurations $\sigma\in *\Sigma_N$ that correspond to a given $\tau\in \Gamma_N'$ in the above way. Thus, the above identity shows that
\[
\sum_{\tau\in \Gamma_N'} e^{\lambda \sum_{p\in P_N} \tau_p} = \frac{e^{-\lambda b_N}}{a_N}\sum_{\sigma\in *\Sigma_N} e^{\lambda\sum_{p\in P_N^*} \sigma_p}.
\]
Plugging this into \eqref{znbeta3} completes the proof of the theorem.
\end{proof}

\subsection{Duality for expected values}\label{duality2}
Consider Ising lattice gauge theory on the cube $B_N$ with inverse coupling strength $\beta$ and free boundary condition. As usual, let $\Sigma_N$ be the space of configurations. Let $f$ be any real-valued function on $\Sigma_N$. We will denote the expected value of $f$ under this theory by $\smallavg{f}_{N, \beta}$. Similarly, if $*\Sigma_N$ is the set of configurations for Ising lattice gauge theory with zero boundary condition on the dual cube $*B_N$, and $f$ is a real-valued function on $*\Sigma_N$, then the expected value of $f$ at inverse coupling strength $\lambda$ will be denoted by~$\smallavg{f}_{N, \lambda}^*$.

For a finite collection of plaquettes $P$, define
\begin{equation}\label{fundef}
\pi_P(\sigma):= \prod_{p\in P} \sigma_p, \quad \psi_P(\sigma) := \sum_{p\in P} \sigma_p. 
\end{equation}
The following proposition relates the expected value of $\pi_P$ under Ising lattice gauge theory with a certain expectation involving $\psi_{*P}$ in the dual model. 
\begin{thm}\label{exppropthm}
Let $P$ be a finite collection of plaquettes and let $*P$ be the set of dual plaquettes of $P$. Let $N$ be so large that $P$ is contained in $B_N$.  Let $\lambda$ be related to $\beta$ as in \eqref{lambdadef}. Then 
\eq{
\smallavg{\pi_P}_{N,\beta} = \smallavg{e^{-2\lambda \psi_{*P}}}_{N,\lambda}^*.
}
\end{thm}
\begin{proof}
Let all notation be as in the proof of Theorem \ref{zpropthm}. Then 
\eq{
\smallavg{\pi_P}_{N,\beta} &= \frac{1}{Z_N(\beta)}\sum_{\sigma\in \Sigma_N} \prod_{p\in P} \sigma_p \prod_{p\in P_N} e^{\beta\sigma_p}.
}
For each $\tau\in \Gamma_N$, define a vector $\tau'\in \{0,1,2\}^{P_N}$ as:
\eq{
\tau'_p =
\begin{cases}
(1-\tau_p)/2 &\text{if } p\not \in P,\\
(1-\tau_p)/2+ 1&\text{if } p\in P.
\end{cases}
}
Proceeding as in the proof of Theorem \ref{zpropthm}, we get
\begin{align}
\sum_{\sigma\in \Sigma_N} \prod_{p\in P} \sigma_p \prod_{p\in P_N} e^{\beta\sigma_p} &= \alpha^{|P_N|}\sum_{\sigma\in \Sigma_N} \Big(\prod_{p\in P} \sigma_p\Big)\Big(\sum_{\tau\in \Gamma_N}\prod_{p\in P_N} \sigma_p^{(1-\tau_p)/2}e^{\lambda \tau_p}\Big)\nonumber\\
&=  \alpha^{|P_N|}\sum_{\sigma\in \Sigma_N} \sum_{\tau\in \Gamma_N}\prod_{p\in P_N} \sigma_p^{\tau_p'}e^{\lambda \tau_p}\nonumber\\
&= \alpha^{|P_N|}\sum_{\tau\in \Gamma_N} e^{\lambda\sum_{p\in P_N} \tau_p}\Big(\sum_{\sigma\in \Sigma_N} \prod_{e\in E_N} \sigma_e^{h(\tau,e)}\Big),\label{expform}
\end{align}
where
\[
h(\tau, e):= \sum_{p\in P_N(e)} \tau'_p.
\] 
Now, for each $\tau\in \Gamma_N$, define another vector $\tau''\in \Gamma_N$ as 
\eq{
\tau''_p &= 
\begin{cases}
\tau_p &\text{if } p\not \in P,\\
-\tau_p &\text{if } p\in P. 
\end{cases}
}
Note that if $p\not \in P$, then $\tau_p'$ is even if and only if $\tau_p = 1$. On the other hand, if $p\in P$, then $\tau_p'$ is even if and only if $\tau_p = -1$. Thus, for any $p$, $\tau'_p$ is even if and only if $\tau''_p=1$. From this observation it follows easily that $h(\tau, e)$ is even if and only if 
\begin{equation}
\prod_{p\in P_N(e)} \tau_p'' = 1.\label{prodcond2}
\end{equation}
Consequently,
\begin{align}
\sum_{\sigma\in \Sigma_N} \prod_{e\in E_N} \sigma_e^{h(\tau,e)}  &= \prod_e (1+ (-1)^{h(\tau,e)}) \nonumber\\
&= 
\begin{cases}
2^{|E_N|} &\text{ if  $\prod_{p\in P_N(e)} \tau_p'' = 1$ for every $e\in E_N$,}\\
0 &\text{ otherwise.}
\end{cases}
\label{sigmaform3}
\end{align}
Let $\Gamma_N''$ be the set of all $\tau\in \Gamma_N$ that satisfy \eqref{prodcond2}. Then by \eqref{expform} and \eqref{sigmaform3}, 
\eq{
\sum_{\sigma\in \Sigma_N} \prod_{p\in P} \sigma_p \prod_{p\in P_N} e^{\beta\sigma_p} &= 2^{|E_N|} \alpha^{|P_N|} \sum_{\tau\in \Gamma_N''} e^{\lambda \sum_{p\in P_N} \tau_p}.
}
Now recall the set $\Gamma_N'$ from the proof of Theorem \ref{zpropthm}. Note that $\tau\in \Gamma_N''$ if and only if $\tau''\in \Gamma_N'$. Moreover, the map $\tau\mapsto \tau''$ is a bijection, which is its own inverse. Thus, 
\eq{
\sum_{\tau\in \Gamma_N''} e^{\lambda \sum_{p\in P_N} \tau_p} &= \sum_{\tau\in \Gamma_N'} e^{\lambda \sum_{p\in P_N} \tau''_p} = \sum_{\tau\in \Gamma_N'} e^{ - 2\lambda \sum_{p\in P} \tau_p + \lambda \sum_{p\in P_N} \tau_p}.
}
From the last part of the proof of Theorem \ref{zpropthm}, we know how to sum over $\tau\in \Gamma_N'$ by identifying $\tau_p$ with $\sigma_p$ for some configuration $\sigma$ on $*\Sigma_N$ that satisfies the zero boundary condition. Using the formula for $Z_N(\beta)$ from Theorem \ref{zpropthm}, this completes the proof.
\end{proof}

\section{Decay of correlations}\label{corsec}
The duality relations of Section \ref{dualsec} allow us to transfer calculations from the weak coupling regime to the strong coupling regime. This is very helpful since the strong coupling regime has exponential decay of correlations. The easiest way to prove exponential decay of correlations at strong coupling is by using Dobrushin's condition~\cite{dobrushin68, dobrushin70}. Although Dobrushin's condition was originally developed for proving the uniqueness of Gibbs states, it was later realized that the criterion can be used to prove exponential decay of correlations. A variant of this result is presented below. Since I could not find the exact statement in the literature, a short proof  is included. 

\subsection{Correlation decay by Dobrushin's condition}
Let $(\Omega, \mf)$ be a measurable space. Recall that the total variation distance between two probability measures $\nu$ and $\nu'$ on $(\Omega, \mf)$ is defined as
\[
\tv(\nu, \nu') := \sup_{A\in \mf} |\nu(A)-\nu'(A)|.
\]
Recall that the total variation distance can also be represented as
\begin{equation}\label{tvcoup}
\inf\pp(X\ne Y),
\end{equation}
where the infimum is taken over all pairs of $\Omega$-valued random variables $(X,Y)$ defined on the same probability space, such that $X$ has law $\nu$ and $Y$ has law $\nu'$. Yet another representation is
\[
\frac{1}{2}\sup\biggl|\int fd\nu - \int f d\nu'\biggr|,
\]
where the supremum is over all measurable $f:\Omega \to [-1,1]$. (For the equivalence of these representations, see~\cite[Exercise 3.6.2]{durrett}.)

Now suppose that $\Omega$ is a complete separable metric space. Take any $n\ge 1$, and let  $\mu$ and $\mu'$ be two probability measures on $\Omega^n$. In the following, we shall use the notation $\bar{x}^i$ to denote the element of $\Omega^{n-1}$ obtained by dropping the $i^{\mathrm{th}}$ coordinate of a vector $x\in \Omega^n$. For any $x\in \Omega$, $\mu_i(\cdot|\bar{x}^i)$ will denote the conditional law (under $\mu$) of the $i^{\mathrm{th}}$ coordinate given that the vector of the remaining coordinates equals $\bar{x}^i$. Define $\mu_i'$ similarly. Since $\Omega$ is a Polish space, regular conditional probabilities exist (see \cite[Section 5.1.3]{durrett}) and therefore our definitions of $\mu_i$ and $\mu_i'$ make sense.  

Suppose that for any $i$ and any $x,y\in \Omega$, the following condition holds:
\begin{equation}\label{tvcond}
\tv(\mu_i(\cdot|\bar{x}^i), \mu^\prime_i(\cdot|\bar{y}^i)) \le \sum_{j=1}^n \alpha_{ij} 1_{\{x_j\ne y_j\}} + h_i,
\end{equation}
where $\alpha_{ij}$'s and $h_i$'s are fixed nonnegative real numbers, and $1_{\{x_j\ne y_j\}} = 1$ if $x_j\ne y_j$ and $0$ otherwise.
Assume that
\begin{equation}\label{dobcond}
s := \max_{1\le i\le n} \sum_{j=1}^n \alpha_{ij} < 1.
\end{equation}
The above assumption is a version of Dobrushin's condition. Let $Q$ be the matrix $(\alpha_{ij})_{1\le i,j\le n}$ and suppose that $P$ is a Markov transition matrix on $\{1,\ldots,n\}$ such that $Q \le sP$ elementwise. For each $i$ and $j$, let $\tau_{ij}$ be the first hitting time of state $j$ starting from state $i$ of a Markov chain with transition kernel $P$. Then the following holds.
\begin{thm}\label{decaythm}
Let all notation be as above, and suppose that \eqref{tvcond} and \eqref{dobcond} hold. Let $Z$ and $Z'$ be two $\Omega^n$-valued random vectors, with laws $\mu$ and $\mu'$ respectively. Take any $A\subseteq \{1,\ldots,n\}$. Let $\nu$ be the law of $(Z_i)_{i\in A}$ and $\nu'$ be the law of $(Z_i')_{i\in A}$. Then
\[
\tv(\nu, \nu') \le \frac{1}{1-s} \sum_{i\in  A} \sum_{j=1}^n \ee(s^{\tau_{ij}})h_j. 
\]
\end{thm}
\begin{proof}
Let $X_0$ and $X_0'$ be independent $\Omega^n$-valued random vectors with laws $\mu$ and $\mu'$ respectively. We will now inductively define a Markov chain $(X_k,X_k')_{k\ge 0}$. Given $(X_k, X_k')=(x,y)$ for some $k$, generate $(X_{k+1}, X'_{k+1})$ as follows. Choose a coordinate $J$ uniformly at random from $\{1,\ldots,n\}$. Generate a pair $(W, W')$  of $\Omega$-valued random variables such that $W$ has law $\mu_J(\cdot|\bar{x}^J)$ and $W'$ has law $\mu'_J(\cdot|\bar{y}^J)$, but their joint distribution is such that
\[
\pp(W\ne W') = \tv(\mu_J(\cdot|\bar{x}^J), \mu'_J(\cdot|\bar{y}^J)).
\]
(This can be done because the infimum is attained in the coupling characterization~\eqref{tvcoup} of total variation distance.) 
Having obtained $W$ and $W'$, produce $X_{k+1}$ by replacing the $J^{\mathrm{th}}$ coordinate of $X_k$ by $W$ (keeping all other coordinates the same) and produce $X_{k+1}'$ similarly using $W'$. It is not hard to see that with this construction, $X_k$ has law $\mu$ and $X_k'$ has law $\mu'$ for every~$k$. 

Let $X_{k,i}$ denote the $i^{\mathrm{th}}$ coordinate of $X_k$. Then the above construction and \eqref{tvcond} imply that for any $i$, 
\begin{align*}
&\pp(X_{k+1,i}\ne X'_{k+1,i} |X_k, \xp_k) \\
&\le \biggl(1-\frac{1}{n}\biggr) 1_{\{X_{k,i}\ne X'_{k,i}\}} + \frac{1}{n} \sum_{j=1}^n \alpha_{ij} 1_{\{X_{k,j} \ne X'_{k,j}\}} + \frac{h_i}{n}.
\end{align*}
Now let $\ell_k$ denote the vector in $\rr^n$ whose $i^{\mathrm{th}}$ coordinate is $\pp(X_{k,i}\ne \xp_{k,i})$. From the above inequality, it follows that 
\[
\ell_{k+1} \le \biggl(1-\frac{1}{n}\biggr) \ell_k + \frac{1}{n} Q\ell_k + \frac{1}{n}h, 
\]
where the inequality means that each coordinate of the vector on the left is dominated by the corresponding coordinate on the right. 
Let 
\[
\ell := \limsup_{k\ra \infty} \ell_k,
\]
 where the  $\limsup$ is taken coordinatewise. By the above inequality, it follows that
\[
\ell \le \biggl(1-\frac{1}{n}\biggr) \ell + \frac{1}{n}Q\ell + \frac{1}{n}h,
\]
which simplifies to
\begin{equation}\label{iqeq}
(I-Q)\ell \le h,
\end{equation}
where $I$ is the $n\times n$ identity matrix.
The condition $Q\le sP$ ensures that $Q$ has spectral norm less than $1$. Therefore,  $I-Q$ is invertible and 
\[
(I-Q)^{-1} = \sum_{k=0}^\infty Q^k.
\]
Moreover, this matrix has nonnegative entries. Hence, by \eqref{iqeq},
\[
\ell \le \sum_{k=0}^\infty Q^k h \le \sum_{k=0}^\infty s^k P^k h.
\]
Let $\{Y_k\}_{k\ge 0}$ be a Markov chain on $\{1,\ldots, n\}$ with transition kernel $P$. If $p_{ij}^{(k)}$ stands for the $(i,j)^{\mathrm{th}}$ entry of $P^k$, then
\begin{align*}
\sum_{k=0}^\infty s^k p_{ij}^{(k)} &= \sum_{k=0}^\infty \ee(s^k 1_{\{Y_k = j\}} |Y_0 = i) \\
&\le \ee\biggl(\sum_{k=\tau_{ij}}^\infty s^k\biggr) = \frac{\ee(s^{\tau_{ij}})}{1-s}. 
\end{align*}
The proof is now easily completed by recalling the coupling characterization~\eqref{tvcoup} of total variation distance. 
\end{proof}

\subsection{Correlation decay at strong coupling}
Consider Ising lattice gauge theory in the dual cube $*B_N$ at inverse coupling strength $\lambda$, under zero boundary condition. Let $\mu$ be the probability measure on $*\Sigma_N$ defined by this theory.  Let $B$ be a subcube of $B_N$. Consider Ising lattice gauge theory on $*B$ at inverse coupling strength $\lambda$, under zero boundary condition. The probability measure defined by this theory can be extended to a probability measure $\mu'$ on $*\Sigma_N$ by defining the spins outside $*B$ to be all equal to $1$. We will now prove the following result by  applying Theorem \ref{decaythm} to the pair of measures $(\mu, \mu')$. 
\begin{thm}\label{decay1}
Let all notation be as above. There are positive universal constants $\lambda_0$, and $C_0$ such that the following holds when $\lambda \le \lambda_0$. Let $A$ be any set of internal edges of $*B$. Suppose that all the vertices of all the edges of $A$ are at a distance at least $l$ from the boundary of $*B$. Let $j$ be the width of $*B$. Let $\nu$ be the law of $(\sigma_e)_{e\in A}$ under $\mu$, and let $\nu'$ be the law of $(\sigma_e)_{e\in A}$ under $\mu'$. Then
\[
\tv(\nu,\nu')\le j^3(C_0\lambda)^{C_0l}.
\]
\end{thm}
\begin{proof}
Take any internal edge $e$ of $*B$. Define the set of neighbors $N(e)$ of~$e$ to be the set of all edges that belong to any one of the plaquettes containing~$e$. There are at most $18$ such neighbors. Following the conventions of the previous subsection, let
\[
\bar{\sigma}_e := (\sigma_{f})_{f\ne e} 
\] 
and let $\mu_e(\cdot|\bar{\sigma}_e)$ be the conditional law of $\sigma_e$ given $\bar{\sigma}_e$ under $\mu$. Define $\mu'_e$ similarly. Then it is not difficult to write down the conditional distributions explicitly and verify  that there is a universal constant $L_0$ such that 
\[
\tv(\mu_e(\cdot|\bar{\sigma}_e), \mu_e'(\cdot|\bar{\sigma}_e')) \le L_0\lambda
\]
for any $\bar{\sigma}_e$ and $\bar{\sigma}_e'$. Moreover, if $\sigma_f=\sigma'_f$ for every $f\in N(e)$, then the above distance is zero. Thus, if we take 
\[
\alpha_{ef} = 
\begin{cases}
L_0\lambda &\text{ if } f\in N(e),\\
0 &\text{ otherwise,}
\end{cases}
\]
and $h_e=0$, then the condition \eqref{tvcond} holds for $i=e$.

On the other hand, if $e$ is an edge of $*B_N$ but not an internal edge of $*B$, then we have the trivial inequality
\[
\tv(\mu_e(\cdot|\bar{\sigma}_e), \mu_e'(\cdot|\bar{\sigma}_e')) \le 1,
\]
which allows us to take $\alpha_{ef}=0$ for all $f$ and $h_e=1$. 

Let $Q= (\alpha_{ef})_{e,f\in *E_N}$, where $*E_N$ is the set of edges of $*B_N$.  With the above definition of $\alpha_{ef}$, it is clear that if we take $P$ to be transition kernel of the simple random walk on $*E_N$ where the jumps happen from an edge to one of its neighbors as defined above, then $Q\le sP$, where 
\[
s := 18L_0\lambda.
\]
Choose $\lambda_0$ so small that $s\le 1/2$ when $\lambda \le \lambda_0$. 
Let $A$ be as in the statement of the theorem. Then note that there is a universal constant $L_1$ such that if we start the above random walk from an edge in $A$,  it must take at least $L_1(l+k)$ steps to reach an edge that is not an internal edge of $*B$ and at least one of whose vertices is at a distance exactly $k$ from the boundary of $*B$. There are at most $L_2(j+k)^3$ such edges, where $L_2$ is another universal constant. By Theorem \ref{decaythm}, this shows that
\begin{align*}
\tv(\nu,\nu') &\le \frac{1}{1-s}\sum_{k=0}^\infty s^{L_1(k+l)}L_2(j+k)^3\\
&\le 16L_2s^{L_1l}\sum_{k=0}^\infty s^{L_1k}(j^3+k^3)\\
&\le 16L_2 (18L_0\lambda)^{L_1l} \sum_{k=0}^\infty 2^{-L_1k} (j^3 + k^3).
\end{align*}
This completes the proof of the theorem.
\end{proof}

\subsection{Correlation decay at weak coupling}
Let $B_N$ and $B$ be as in the previous subsection. As before, let $\smallavg{\cdot}_{N, \beta}$ denote expectation with respect to Ising lattice gauge theory on $B_N$ at inverse coupling strength $\beta$ and free boundary condition. Let  $\smallavg{\cdot}_{B, \beta}$ denote expectation with respect to Ising lattice gauge theory on $B$ at inverse coupling strength $\beta$ and free boundary condition. Then we have the following result.
\begin{thm}\label{decay2}
Let all notation be as above. There are positive universal constants $\beta_0$  and $C_0$ such that the following holds when $\beta \ge \beta_0$. 
Let $P$ be a set of plaquettes of $B$, whose vertices are at a distance at least $l$ from the boundary of $B$. Let $j$ be the width of $B$. Let $f$ be a any function of $(\sigma_p)_{p\in P}$. Then 
\begin{equation}\label{decay2bd}
|\smallavg{f}_{N,\beta} - \smallavg{f}_{B,\beta}|\le C_f j^3 (C_0 e^{-2\beta})^{C_0l},
\end{equation}
where $C_f$ is a constant that depends only on $f$ (and implicitly on $P$, since $f$ is a function on $P$). 
\end{thm}
\begin{proof}
By the Fourier--Walsh expansion for functions of binary variables (see \cite[Chapter 1]{odonnell}), any function of $(\sigma_p)_{p\in P}$ can be written as a linear combination of $(\pi_Q)_{Q\subseteq P}$. Thus, it suffices to prove the theorem for $f = \pi_P$. By Theorem \ref{exppropthm}, 
\[
\smallavg{\pi_P}_{N,\beta} = \smallavg{e^{-2\lambda\psi_{*P}}}^*_{N,\lambda},
\]
where $\lambda$ is related to $\beta$ as in \eqref{lambdadef}. Similarly, if $\smallavg{\cdot}^*_{B,\lambda}$ denotes expectation with respect to Ising lattice gauge theory on $*B$ at inverse coupling strength $\lambda$ and zero boundary condition, then 
\[
\smallavg{\pi_P}_{B,\beta} = \smallavg{e^{-2\lambda\psi_{*P}}}^*_{B,\lambda}.
\]
Let $\lambda_0$ be as in Theorem \ref{decay1}, and choose $\beta_0$ such that $\beta\ge \beta_0$ if and only if $\lambda \le \lambda_0$. Suppose that $\beta \ge \beta_0$. Then by Theorem \ref{decay1}, 
\[
|\smallavg{e^{-2\lambda\psi_{*P}}}^*_{N,\lambda} - \smallavg{e^{-2\lambda\psi_{*P}}}^*_{B,\lambda}| \le e^{2\lambda_0|P|} j^3 (C_0\lambda)^{C_0l}.
\]
By the definition of $\lambda$, it follows that $\lambda \le Ce^{-2\beta}$ for some universal constant $C$. Plugging this into the above inequality, we get the desired result. 
\end{proof}
\subsection{Existence of the infinite volume limit}\label{infproofsec}
In this subsection, we will prove the existence of the limit \eqref{wilsonexplim}. We will first prove the following stronger result, and then deduce the existence of the limit as a corollary.
\begin{thm}\label{infthm2}
There is some $\beta_0$ such that whenever $\beta\ge \beta_0$, the following holds. Let $P$ be any finite set of plaquettes, and let $f$ be a function of $(\sigma_p)_{p\in P}$. Then $\smallavg{f}_{N,\beta}$ converges to a limit as $N\to \infty$. Moreover, this limit is translation invariant, in the sense that if $g$ is the composition of a translation of the lattice followed by $f$, then $\smallavg{g}_{N,\beta}$ converges to the same limit as $N\to\infty$. 
\end{thm}
\begin{proof}
Let $B_N$ and $B$ be as in Theorem \ref{decay2}, so that we get the bound~\eqref{decay2bd}. Replacing $N$ be a larger number $N'$, we get the same bound, because the bound does not depend on $N$. Thus, the difference between $\smallavg{f}_{N,\beta}$ and $\smallavg{f}_{N',\beta}$ is bounded by twice the quantity on the right side of \eqref{decay2bd}. Since we can choose a large $B$ if $N$ and $N'$ are large, this shows that the sequence $(\smallavg{f}_{N,\beta})_{N\ge 1}$ is a Cauchy sequence. This proves the existence of the limit. Let $\alpha$ denote the limit. To prove translation invariance, simply revisit \eqref{decay2bd} and send $N$ to infinity. This shows that
\[
|\alpha - \smallavg{f}_{B,\beta}|\le C_f  j^3 (C_0 e^{-2\beta})^{C_0l},
\]
where $j$ is the width of $B$. 
Therefore, as $B$ increases in size, $\smallavg{f}_{B,\beta}$ tends to $\alpha$. 
Let $g$ be the composition of a translation of the lattice followed by $f$. Let $B'$ be the same translation applied to the cube $B$. Then clearly,
\[
\smallavg{f}_{B,\beta} = \smallavg{g}_{B',\beta}.
\]
As $B$ increases in size, so does $B'$, and both quantities in the above display converge to their infinite volume limits. This completes the proof of translation invariance. 
\end{proof}
\begin{cor}\label{infcor}
There exists $\beta_0>0$ such that if $\beta \ge \beta_0$, then for any generalized loop $\gamma$ in $\zz^4$, the limit $\smallavg{W_\gamma}_\beta := \lim_{N\to\infty} \smallavg{W_\gamma}_{N,\beta}$ 
exists. Moreover, the limit is translation invariant, in the sense that the Wilson loop variable for a translate of $\gamma$ will have the same limiting expected value.
\end{cor}
\begin{proof}
By Lemma \ref{looplmm}, any generalized loop $\gamma$ is the boundary of a surface $P$ of finite size. Consequently, 
\[
W_\gamma = \prod_{p\in P} \sigma_p.
\]
The proof is now easily completed using Theorem \ref{infthm2}.
\end{proof}

\section{Distribution of vortices}\label{vorsec}
In this section, we will make some deductions about the distribution of vortices. Throughout this section, we will fix a $\beta \ge \beta_0$, where $\beta_0$ is as in Theorem \ref{decay2}. We will also fix an infinite volume limit $\mu_\beta$ of Ising lattice gauge theory on $\zz^4$ with free boundary condition at inverse coupling strength $\beta$, obtained by taking a subsequential limit of the probability measures on $\Sigma_N$ as $N\to\infty$. 

To be more precise, let $\Sigma$ be the set of all spin configurations on the edges of $\zz^4$. Equip $\Sigma$ with the product topology and the cylinder $\sigma$-algebra generated by this topology. Under the product topology, $\Sigma$ is a compact metric space. Consequently, any sequence of probability measures on $\Sigma$ has a subsequential weak limit. 
Consider now Ising lattice gauge theory with free boundary condition on $B_N$ at inverse coupling strength $\beta$. The probability measure on $\Sigma_N$ defined by this theory can be lifted to a probability measure on $\Sigma$ by fixing the spins outside $B_N$ to be all equal to $1$. The compactness argument given above allows us to extract a subsequential limit of this sequence of probability measures, which we call $\mu_\beta$.

By Theorem \ref{decay2}, we know that for a function of finitely many plaquettes, the expected value under $\mu_\beta$ does not depend on our choice of $\mu_\beta$. This fact will be invoked implicitly on several occasions. 

\subsection{Rarity of negative plaquettes}
Given a configuration $\sigma$, recall that we call a plaquette $p$ negative if $\sigma_p = -1$. When $\beta$ is large, one may expect that negative plaquettes are rare. The following theorem gives a quantitative estimate for this.
\begin{thm}\label{rarity}
Let $P$ be any finite set of plaquettes. There is a constant $C_P$ depending only on $P$ such that under $\mu_\beta$, the probability of the event that $\sigma_p=-1$ for all $p\in P$ is bounded above by $C_Pe^{-2\beta|P|}$. 
\end{thm}
\begin{proof}
Throughout this proof, $C_P$ will denote any constant that depends only on $P$.  The value of $C_P$ may change from line to line. 

Let $\rho$ be the probability of the stated event under $\mu_\beta$. Recall the universal constant $C_0$ from Theorem \ref{decay2}. Take a cube  $B$ containing $P$, whose width is the smallest required to ensure that the distance of any vertex of any plaquette in $P$ to the boundary of $B$ is at least $|P|/C_0$. Let $\rho'$ be the probability of the stated event in the Ising lattice gauge theory on $B$ at inverse coupling strength $\beta$ and free boundary condition. Then by Theorem~\ref{decay2},
\begin{equation}\label{rhorho}
|\rho-\rho'|\le C_P e^{-2\beta |P|}. 
\end{equation}
Let $Z_B(\beta)$ be the normalizing constant for the Ising lattice gauge theory on $B$. The configuration of all plus spins contributes $e^{\beta |P(B)|}$ to this normalizing constant, where $P(B)$ is the set of plaquettes of $B$. Thus, 
\[
Z_B(\beta)\ge e^{\beta|P(B)|}. 
\]
On the other hand, any configuration that has $\sigma_p=-1$ for all $p\in P$ contributes at most $e^{\beta(|P(B)|-2|P|)}$. The number of such configurations depends only on $B$. Since the width of $B$ is determined by $P$, the number of such configurations depends only on $P$. Combining these observations, we get
\[
\rho'\le C_P e^{-2\beta |P|}. 
\]
Combined with \eqref{rhorho}, this completes the proof.
\end{proof}
\subsection{Rarity of large vortices}
An important corollary of the Theorem \ref{rarity} is the following, which shows that large vortices are rare. Incidentally, it also shows that with probability one, there are no infinite vortices.
\begin{cor}\label{vorcor}
For any plaquette $p$, the probability under $\mu_\beta$ that $p$ is contained in a vortex of size $\ge m$ is bounded above by $C_m e^{-2\beta m}$, where $C_m$ is a constant that depends only on $m$. 
\end{cor}
\begin{proof}
Suppose that $p$ is contained in a vortex $P$ of size $\ge m$. Since $*P$ is a connected set of plaquettes, it follows that there are at least $m$ negative plaquettes within a distance $C_m'$ from $p$, where $C_m'$ depends only on $m$. Let $Q$ be any collection of $m$ plaquettes within this distance from $p$. By Theorem \ref{rarity}, the chance that these plaquettes are all negative is bounded by $C_Qe^{-2\beta m}$, where $C_Q$ depends only on $Q$. The proof is completed by summing over all choices of $Q$.
\end{proof}
A consequence of Corollary \ref{vorcor} is the following result, which shows that there are no large vortices is small regions.
\begin{cor}\label{vorcor2}
Under $\mu_\beta$, the chance that there is a vortex of size $\ge m$ intersecting a set of plaquettes of size $j$ is bounded above by $C_mje^{-2\beta m}$, where $C_m$ depends only on $m$. 
\end{cor}
\begin{proof}
This is just a union bound,  applying Corollary \ref{vorcor} to each plaquette and summing over all plaquettes.
\end{proof}

\section{Proof of the main result}\label{proofsec}
Throughout this subsection, we will fix some $\beta \ge \beta_0$ as in Theorem \ref{decay2} and an infinite volume Gibbs measure $\mu_\beta$ as in Section \ref{vorsec}. We will use the notation $\mu_\beta(A)$ to denote the probability of an event $A$ under $\mu_\beta$, and $\mu_\beta(X)$ to denote the expected value of a random variable $X$ under $\mu_\beta$. Also, throughout, $C$ will denote any universal constant. The value of $C$ may change from line to line or even within a line.

The proof of Theorem \ref{genthm} is divided into two lemmas, to deal with the cases of short loops and long loops separately. The first lemma gives a bound that is useful when $\ell$ is of order $e^{12\beta}$ or less.
\begin{lmm}\label{genlmm1}
In the setting of Theorem \ref{genthm}, we have
\[
|\smallavg{W_\gamma}_\beta -e^{-2\ell e^{-12\beta}}|\le Ce^{C\ell e^{-12\beta}}\biggl(e^{-2\beta}+\sqrt{\frac{\ell_0}{\ell}}\biggr).
\]
\end{lmm}
The proof of the above lemma is somewhat lengthy, so it is divided into a number of steps. The main idea is to reduce the Wilson loop variable $W_\gamma$ to something more manageable. We will successively reduce $W_\gamma$ to the variables $W_\gamma^0, W_\gamma^3,W_\gamma^4, W_\gamma^5,W_\gamma^6$ of increasing simplicity, such that the variables are all equal to each other with high probability. The simplest variable, $W_\gamma^6$, has the form $(-1)^{N_6}$ where $N_6$ is an approximately binomial random variable with mean $\ell e^{-12\beta}$. This is the key fact that allows us to show that the expected value of $W_\gamma$ is approximately $e^{-2\ell e^{-12\beta}}$. To satisfy the reader's curiosity, let us briefly mention right away that $N_6$ is the number of non-corner edges $e\in \gamma$ such that each plaquette containing $e$ is negative.


Let us define the variables $W_\gamma^0, W_\gamma^3,W_\gamma^4, W_\gamma^5,W_\gamma^6$ before proving anything about them. Since $\gamma$ has length $\ell$, there is a cube $B$ of width $\ell$ that contains $\gamma$.  By Lemma~\ref{looplmm}, there is a surface $Q$ inside $B$ such that $\gamma$ is the boundary of $Q$. Fix $Q$ and $B$. For simplicity of notation, let
\[
\alpha := \ell e^{-12\beta}, \ \  r := \frac{\ell_0}{\ell}.
\]
Let $\sigma$ be a random configuration drawn from the Gibbs measure $\mu_\beta$. Let $V$ be the set of vortices of $\sigma$ that intersect $Q$. Let 
\[
N := \sum_{P\in V} |P\cap Q|.
\]
Let $V_0$ be the set of all members of $V$ that have size $\le 24$. Let
\[
N_0 := \sum_{P\in V_0} |P\cap Q|,
\]
and define 
\[
W_\gamma^0 := (-1)^{N_0}. 
\]
Let $b$ be the smallest number such that any vortex of size $\le 24$ is contained in a cube of width $b$. By the definition of vortex, $b$ is a finite universal constant. Let $Q'$ be the set of plaquettes $p\in Q$ that are so far away from $\gamma$ that any cube of width $b+2$ containing $p$ does not intersect $\gamma$. Note that $Q'$ may be empty. Let $V_1$ be the set of all $P\in V_0$ that intersect $Q'$. Let
\[
N_1 := \sum_{P\in V_1} |P\cap Q|,
\]
with the usual convention that an empty sum equals zero. Let $V_2 := V_0\setminus V_1$ be the set of members of $V_0$ that do not intersect $Q'$, and let 
\[
N_2 := \sum_{P\in V_2}|P\cap Q| =N_0-N_1.
\]
Geometrically, $Q\setminus Q'$ is the part of $Q$ that is `close to the boundary $\gamma$'. Thus, $N_2$ counts the number of negative plaquettes coming from vortices (of size $\le 25$) that are close to the boundary. 

Let us denote the set of plaquettes containing an edge $e$ by $P(e)$.
It is easy to see that the smallest possible closed connected surface is a set of plaquettes bounding a $3$-cell. If $R$ is such a surface, bounding a $3$-cell $c$, then $*R = P(e)$, where the edge $e$ is the dual of  $c$.  Thus, the smallest possible size of a vortex is $6$, and any such vortex must be $P(e)$ for some $e$. Such a vortex will be called a minimal vortex. Let $V_3$ be the set of minimal vortices that intersect $Q$ but not $Q'$, and let
\[
N_3:=\sum_{P\in V_3} |P\cap Q|.
\]
Define
\[
W_\gamma^3 := (-1)^{N_3}.
\]
Next, let $V_4$ be the set of all members of $V_3$ that are $P(e)$ for some $e\in \gamma$. Let
\[
N_4 := \sum_{P\in V_4} |P\cap Q|,
\]
and let
\[
W_\gamma^4 := (-1)^{N_4}.
\]
Let $V_5$ be the set of all members of $V_4$ that are $P(e)$ for some non-corner edge $e$. Let
\[
N_5 := \sum_{P\in V_5} |P\cap Q|
\]
and define
\[
W_\gamma^5 := (-1)^{N_5}. 
\]
Finally, let $N_6$ be the number of non-corner edges $e\in \gamma$ such that all members of $P(e)$ are negative plaquettes (we will henceforth abbreviate this condition as `$P(e)$ is negative'), and let 
\[
W_\gamma^6 := (-1)^{N_6}. 
\]
We will now prove a sequence of simple lemmas to show that $W_\gamma$ is equal to $W_\gamma^6$ with high probability. The proof will proceed roughly as follows. We will begin with the observation that $W_\gamma = (-1)^N$. We will then show that with high probability, no vortex of size $\ge 25$ intersects $Q$. This implies that with high probability, $N=N_0$ and hence $W_\gamma=W_\gamma^0$. The next step is a crucial geometric argument, based on Lemma \ref{evenlmm}, showing that $N_1 = N_0-N_2$ is even. This implies that $W_\gamma^0=(-1)^{N_2}$. Next, we will prove that with high probability, only minimal vortices intersect $Q$ but not $Q'$. This will imply that with high probability, $N_2 = N_3$, and hence $W_\gamma^0 = W_\gamma^3$. After this, we will prove that $N_3 - N_4$ is even, which implies that $W_\gamma^3 = W_\gamma^4$. The final steps will involve showing that $W_\gamma^4 = W_\gamma^5 = W_\gamma^6$ with high probability, which happens because the fraction of corner edges is small, and with high probability, only minimal vortices intersect $Q$ but not $Q'$.
\begin{lmm}\label{a1problmm}
Let $A_1$ be the event that there is no vortex of size $\ge 25$ in the configuration $\sigma$ that intersects $B$. Then
\[
\mu_\beta(A_1)\ge 1-C\alpha^4 e^{-2\beta}.
\]
\end{lmm}
\begin{proof}
By Corollary \ref{vorcor2}, $\mu_\beta(A_1)\ge 1-C\ell^4 e^{-50\beta}$. But $\ell^4 e^{-50\beta} = \alpha^4 e^{-2\beta}$. 
\end{proof}
\begin{lmm}\label{ww0lmm} 
$\mu_\beta(|W_\gamma-W_\gamma^0|) \le C\alpha^4e^{-2\beta}$.
\end{lmm}
\begin{proof}
By Lemma \ref{vorlmm}, $N$ is the number of negative plaquettes of $Q$, and hence
\[
W_\gamma = \prod_{p\in Q} \sigma_p = (-1)^{N}.
\]
If the event $A_1$ happens, then $N_0=N$. Therefore, by Lemma~\ref{a1problmm}, 
\[
\mu_\beta(|W_\gamma-W_\gamma^0|) \le 2(1-\mu_\beta(A_1))\le C\alpha^4e^{-2\beta},
\]
which completes the proof.
\end{proof}

\begin{lmm}\label{wn2lmm}
$W_\gamma^0 = (-1)^{N_2}$. 
\end{lmm}
\begin{proof}
Take any $P\in V_1$.  Then by the definition of $Q'$ and $V_1$, it follows that any cube $B$ of width $b$ that contains $P$ has the property that ${*{*B}}\cap Q$ contains only internal plaquettes of $Q$.  Therefore by Lemma \ref{evenlmm}, $|P\cap  Q|$ is even. Consequently, $N_1$ is even. Since $N_2 = N_0-N_1$ and $W_\gamma^0=(-1)^{N_0}$, this completes the proof.
\end{proof}

\begin{lmm}\label{a2ineqlmm}
Let $A_2$ be the event that no vortex of size $\ge 7$ intersects $Q\setminus Q'$. Then $\mu_\beta(A_2)\ge 1-C\alpha e^{-2\beta}$. 
\end{lmm}
\begin{proof}
Since each plaquette of $Q\setminus Q'$ is contained in a cube of width $b+2$ which intersects $\gamma$, it follows that
\[
|Q\setminus Q'|\le C\ell.
\]
Therefore by Corollary \ref{vorcor2}, $\mu_\beta(A_2)\ge 1-C\ell e^{-14\beta} \ge 1-C\alpha e^{-2\beta}$. 
\end{proof}

\begin{lmm}\label{w0w3lmm}
$\mu_\beta(|W_\gamma^0-W_\gamma^3|) \le C\alpha e^{-2\beta}$.
\end{lmm}
\begin{proof}
If the event $A_2$ happens, then $V_3 = V_2$ and hence $N_3=N_2$. Therefore by Lemma~\ref{wn2lmm}, $A_2$ implies that $W_\gamma^0=W_\gamma^3$. Consequently, by Lemma~\ref{a2ineqlmm},  $\mu_\beta(|W_\gamma^0-W_\gamma^3|) \le 2(1-\mu_\beta(A_2))\le C\alpha e^{-2\beta}$.
\end{proof}

\begin{lmm}\label{w3w4lmm}
$W_\gamma^3 = W_\gamma^4$.
\end{lmm}
\begin{proof}
For any $e$, $P(e)\cap Q$ is the number of plaquettes of $Q$ that contain $e$. In particular, if $e$ is an internal edge of $Q$, then $|P(e)\cap Q|$ is even by the definition of internal edge in Section \ref{strucsec}, since $e$ is contained in an even number of plaquettes of $Q$ and this set is exactly $P(e)\cap Q$. Thus, $N_3-N_4$ is even, and hence $W_\gamma^3 = W_\gamma^4$.
\end{proof}

\begin{lmm}\label{w4w5lmm}
$\mu_\beta(|W_\gamma^4-W_\gamma^5|) \le Cr\alpha$.
\end{lmm}
\begin{proof}
Let $A_3$ be the event that there is no corner edge $e$ such that $P(e)$ is a  vortex.  
By Corollary \ref{vorcor2}, 
\[
\mu_\beta(A_3)\ge 1- C\ell_0e^{-12\beta}.
\]
If $A_3$ happens, then $N_4 = N_5$, and hence 
\[
\mu_\beta(|W_\gamma^4-W_\gamma^5|) \le 2(1-\mu_\beta(A_3)) \le C\ell_0e^{-12\beta} \le Cr\alpha.
\]
This completes the proof of the lemma.
\end{proof}

\begin{lmm}\label{w5lmm}
$W_\gamma^5 = (-1)^{|V_5|}$.
\end{lmm}
\begin{proof}
Notice that if $P\in V_5$, then $P\cap Q$ is odd, and hence $(-1)^{|P\cap Q|} = -1$. 
Thus,
\[
W_\gamma^5 = (-1)^{N_5}= \prod_{P\in V_5} (-1)^{|P\cap Q|} = (-1)^{|V_5|},
\]
completing the proof.
\end{proof}

\begin{lmm}\label{w5w6lmm}
$\mu_\beta(|W_\gamma^5-W_\gamma^6|) \le C\alpha e^{-2\beta}$.
\end{lmm}
\begin{proof}
If the event $A_2$ (from Lemma \ref{a2ineqlmm}) happens, then $N_6=|V_5|$. Therefore by Lemmas~\ref{a2ineqlmm} and~\ref{w5lmm}, $\mu_\beta(|W_\gamma^5-W_\gamma^6|) \le 2(1-\mu_\beta(A_2))\le C\alpha e^{-2\beta}$.
\end{proof}
We will now calculate $\mu_\beta(W_\gamma^6)$ to first order. The key intuition in the following calculation is that $N_6$ is approximately a binomial random variable with mean $\ell e^{-12\beta}$. 
\begin{lmm}\label{w5thetalmm}
Let $\theta := \tanh 6\beta$. Then 
\[
|\mu_\beta(W_\gamma^6) -\theta^{\ell}| \le \frac{Ce^{C\alpha}}{\sqrt{\ell}} + Cre^{C\alpha}. 
\]
\end{lmm}
\begin{proof}
Let $\gamma_1$ be the set of all non-corner edges of $\gamma$, and let $\gamma'$ be the set of edges $e\in \gamma_1$ such that there is at least one positive plaquette and at least one negative plaquette in $P(e)$. In other words, there are some $p,p'\in P(e)$ such that $\sigma_p\sigma_{p'}=-1$. But the product $\sigma_p \sigma_{p'}$ has no dependence on $\sigma_e$. Thus, we can determine whether $e\in \gamma'$ without knowing the value of $\sigma_e$, as long as we know the value of $\sigma_f$ for every edge $f\ne e$ that belongs to some $p \in P(e)$. Since a non-corner edge of $\gamma$ does not share a a plaquette with any other edge of $\gamma$, this implies that we can determine the set $\gamma'$ simply by knowing the values of $\sigma_e$ for all $e\not\in \gamma_1$. 

Let $\mu_\beta'$ denote conditional probability and conditional expectation given  $(\sigma_e)_{e\not\in \gamma_1}$. Since no two non-corner edges belong to the same plaquette, it follows that under this conditioning, $(\sigma_e)_{e\in \gamma_1}$ are independent spins. Moreover, the above paragraph shows that conditioning on the spins outside $\gamma_1$ determines $\gamma'$. 

If $e\in \gamma_1\setminus \gamma'$, a simple calculation gives
\[
\mu_\beta'(\text{$P(e)$ is negative}) =\frac{e^{-6\beta}}{e^{6\beta} + e^{-6\beta}}. 
\]
On the other hand, if $e\in \gamma'$, then
\[
\mu_\beta'(\text{$P(e)$ is negative}) = 0.
\]
Thus, by the conditional independence of $(\sigma_e)_{e\in \gamma_1}$, 
\eq{
\mu_\beta'(W_\gamma^6)  = \theta^{|\gamma_1\setminus \gamma'|},
}
where $\theta = \tanh 6\beta$. This gives
\begin{equation}\label{semifinal}
\mu_\beta(W_\gamma^6)= \mu_\beta(\mu_\beta'(W_\gamma^6)) = \theta^{|\gamma_1|} \mu_\beta(\theta^{-|\gamma'|}).
\end{equation}
Now,
\begin{equation}\label{theta1}
1-Ce^{-12\beta}\le \theta\le 1.
\end{equation}
This shows that for any $0\le j\le \ell$,
\begin{equation}\label{theta2}
\theta^{-j}\le \theta^{-\ell}\le (1-Ce^{-12\beta})^{-\ell}\le e^{C\ell e^{-12\beta}} \le e^{C\alpha},
\end{equation}
assuming that $\beta_0$ is large enough. Consequently, for any $1\le j\le \ell$,
\begin{align}
\theta^{-j}-1&\le j\theta^{-(j-1)}(\theta^{-1}-1)\nonumber\\
&\le Cje^{C\alpha} e^{-12\beta}\nonumber\\
&\le \frac{Cj\alpha e^{C\alpha}}{\ell}\le \frac{Cje^{C\alpha}}{\ell},\label{theta3}
\end{align}
where we used the inequality $xe^x \le Ce^{2x}$ in the last step. By \eqref{semifinal}, \eqref{theta1} and \eqref{theta3}, we get
\begin{align}
|\mu_\beta(W^6_\gamma)-\theta^{\ell}|&\le |\theta^{|\gamma_1|}\mu_\beta(\theta^{-|\gamma'|}) - \theta^\ell|\nonumber \\
&= |\theta^{|\gamma_1|}(\mu_\beta(\theta^{-|\gamma'|}) -1) + \theta^{\ell-\ell_0}- \theta^\ell|\nonumber \\
&\le\theta^{|\gamma_1|}|\mu_\beta(\theta^{-|\gamma'|}) -1| + \theta^\ell(\theta^{-\ell_0}-1)\nonumber\\
&\le |\mu_\beta(\theta^{-|\gamma'|}) -1| + Cre^{C\alpha}. \label{semi2}
\end{align}
Now recall that any negative plaquette is contained in a vortex  and any vortex has size at least $6$. Therefore by Corollary \ref{vorcor}, the probability that any given plaquette is negative (under $\mu_\beta$) is bounded above by $Ce^{-12\beta}$. Consequently,
\begin{align}\label{fmombd}
\mu_\beta(|\gamma'|) \le C\ell e^{-12\beta} \le C\alpha. 
\end{align}
Thus, for any $j>0$,
\eq{
|\mu_\beta(\theta^{-|\gamma'|}) - \mu_\beta(\theta^{-|\gamma'|}1_{\{|\gamma'|\le j\}})| &= \mu_\beta(\theta^{-|\gamma'|}1_{\{|\gamma'|>j\}})\\
&\le \theta^{-\ell} \mu_\beta(\{|\gamma'|>j\})\\
&\le \theta^{-\ell} \frac{C\alpha}{j}\le \frac{C\alpha e^{C\alpha}}{j}\le \frac{Ce^{C\alpha}}{j},
}
where the last two inequalities follow by \eqref{theta2} and $xe^x\le Ce^{2x}$. 
On the other hand, by \eqref{theta1}, \eqref{theta3} and \eqref{fmombd},  
\begin{align*}
|\mu_\beta(\theta^{-|\gamma'|}1_{\{|\gamma'|\le j\}}) - 1| &\le \mu_\beta((\theta^{-|\gamma'|} - 1)1_{\{|\gamma'|\le j\}}) + \mu_\beta(\{|\gamma'|>j\})\\
&\le  \theta^{-j}-1+ \frac{C\alpha}{j} \le \frac{Cje^{C\alpha}}{\ell}  +\frac{C\alpha}{j}.
\end{align*}
Combining the above inequalities and choosing $j=\sqrt{\ell}$, we get
\[
|\mu_\beta(\theta^{-|\gamma'|})-1|\le \frac{Ce^{C\alpha}}{\sqrt{\ell}}. 
\]
Combining this with \eqref{semi2} gives the desired inequality.
\end{proof}
We are now ready to prove Lemma \ref{genlmm1}. 
\begin{proof}[Proof of Lemma \ref{genlmm1}]
Combining Lemmas~\ref{ww0lmm}, \ref{w0w3lmm}, \ref{w3w4lmm}, \ref{w4w5lmm}, \ref{w5w6lmm} and \ref{w5thetalmm}, we have
\[
|\mu_\beta(W_\gamma)-\theta^\ell|\le C(\alpha + \alpha^4)e^{-2\beta} + Cr\alpha +  \frac{Ce^{C\alpha}}{\sqrt{\ell}} + Cre^{C\alpha}.
\]
Since  $\alpha$ and $\alpha^4$ are bounded by $Ce^{C\alpha}$ and $\ell^{-1}\le r$, this simplifies to
\[
|\mu_\beta(W_\gamma)-\theta^\ell|\le Ce^{C\alpha}(e^{-2\beta} + \sqrt{r}). 
\]
To complete the proof, note that as $\beta \to \infty$,
\eq{
\theta &= \frac{e^{6\beta}-e^{-6\beta}}{e^{6\beta}+e^{-6\beta}}= 1 - \frac{2e^{-12\beta}}{1+e^{-12\beta}} \\
&= 1-2e^{-12\beta} + O(e^{-24\beta})\\
&= e^{-2e^{-12\beta}} + O(e^{-24\beta}).
}
Thus,
\[
|\theta-e^{-2e^{-12\beta}}|\le Ce^{-24\beta}. 
\]
By the inequality $|a^\ell-b^\ell|\le \ell|a-b|$ for $a,b\in [0,1]$, this gives
\eq{
|\theta^\ell - e^{-2\ell e^{-12\beta}}| \le C\ell e^{-24\beta}\le C\alpha e^{-12\beta}. 
}
This completes the proof of Lemma \ref{genlmm1}.
\end{proof}
The next lemma is useful when $\ell \gg e^{12\beta}$.
\begin{lmm}\label{genlmm2}
In the setting of Theorem \ref{genthm}, we have
\[
|\smallavg{W_\gamma}_{\beta}|\le e^{-C(\ell-\ell_0) e^{-12\beta}}. 
\]
\end{lmm}
\begin{proof}
As in the proof of Lemma \ref{genlmm1}, let $\gamma_1$ be the set of all non-corner edges of $\gamma$, and let $\mu_\beta'$ denote conditional probability and conditional expectation given  $(\sigma_e)_{e\not\in \gamma_1}$. As observed earlier, under this conditioning, $(\sigma_e)_{e\in \gamma_1}$ are independent spins. 

Take any $e\in \gamma_1$. Then it is easy to see that
\[
\mu_\beta'(\text{$P(e)$ is negative}) = \frac{e^{-m\beta}}{e^{m\beta} + e^{-m\beta}},
\]
where $m$ is an integer between $-6$ and $6$, depending on the spins on the edges of the plaquettes that contain $e$ (other than $e$ itself). Therefore
\begin{align*}
|\mu_\beta'(\sigma_e)| &= \biggl|\frac{e^{m\beta} - e^{-m\beta}}{e^{m\beta} + e^{-m\beta}}\biggr|\\
&= \frac{e^{|m|\beta} - e^{-|m|\beta}}{e^{|m|\beta} + e^{-|m|\beta}} \\
&= 1-\frac{2}{1 + e^{2|m|\beta}}\le 1-\frac{2}{1 + e^{12\beta}}. 
\end{align*}
By conditional independence of the $\sigma_e$'s, this gives
\begin{align*}
|\mu_\beta'(W_\gamma) |&= \prod_{e\in \gamma_1}|\mu_\beta'(\sigma_e)|\le \biggl(1-\frac{2}{1 + e^{12\beta}}\biggr)^{\ell-\ell_0}.
\end{align*}
The proof is now completed by applying the inequality $1-x\le e^{-x}$.
\end{proof}

Finally, we are ready to prove Theorem \ref{genthm} by combining Lemma \ref{genlmm1} and Lemma \ref{genlmm2}. 
\begin{proof}[Proof of Theorem \ref{genthm}]
First, suppose that $\ell_0\le \ell/2$. Then by Lemma \ref{genlmm2}, 
\[
|\smallavg{W_\gamma}_\beta -e^{-2\ell e^{-12\beta}}|\le 2e^{-C_1\ell e^{-12\beta}}
\]
for some universal constant $C_1$. On the other hand, by Lemma \ref{genlmm1}, 
\[
|\smallavg{W_\gamma}_\beta -e^{-2\ell e^{-12\beta}}|\le C_2e^{C_2\ell e^{-12\beta}}\biggl(e^{-2\beta}+\sqrt{\frac{\ell_0}{\ell}}\biggr)
\]
for some other constant $C_2$. Combining these inequalities, we get
\begin{align*}
&|\smallavg{W_\gamma}_\beta -e^{-2\ell e^{-12\beta}}|^{1+C_2/C_1}\\
&\le C_2e^{C_2\ell e^{-12\beta}}\biggl(e^{-2\beta}+\sqrt{\frac{\ell_0}{\ell}}\biggr) (2e^{-C_1\ell e^{-12\beta}})^{C_2/C_1}\\
&= 2^{C_2/C_1}C_2 \biggl(e^{-2\beta}+\sqrt{\frac{\ell_0}{\ell}}\biggr). 
\end{align*}
This proves the claim when $\ell_0\le \ell/2$. If $\ell_0>\ell/2$, then the bound is automatic, since $|\smallavg{W_\gamma}_\beta|\le 1$. 
\end{proof}

\section*{Acknowledgments}
I thank Erik Bates and Jafar Jafarov for carefully checking the proofs, and  Persi Diaconis, Erhard Seiler, Scott Sheffield, Steve Shenker and Edward Witten for valuable conversations and comments. I also thank the anonymous referee for a large number of useful comments.

\end{document}